\documentclass[a4paper, 12pt,oneside,reqno]{amsart}
\usepackage[a4paper]{geometry}
\usepackage[matrix,arrow,curve,cmtip]{xy}

\usepackage{txfonts}
\DeclareMathAlphabet{\mathcal}{OMS}{cmsy}{m}{n} 

\usepackage{
  amsmath,  
  amssymb,  
  amsthm,   
  tikz,     
  fullpage, 
  youngtab, 
  ytableau, 
  thmtools, 
  hyperref, 
  cite,     
  url       
}

\usepackage{amscd}
\usepackage{euscript}
\usepackage{latexsym}
\usepackage[cp1251]{inputenc}
\usepackage[english]{babel}
\usepackage{mathrsfs}
\usepackage{graphicx}
\usepackage{keyval}
\usepackage{mathtools}
\usepackage{tikz}
\usetikzlibrary{arrows,shapes,snakes,automata,backgrounds,petri,through,positioning}
\usetikzlibrary{intersections}
\usepackage[symbol*]{footmisc}
\usepackage{verbatim}

\usepackage{etoolbox}
\patchcmd{\thebibliography}{\section*}{\paragraph}{}{}

\newtheorem{theorem}{Theorem}[section]
\newtheorem{lemma}[theorem]{Lemma}
\newtheorem{proposition}[theorem]{Proposition}
\newtheorem{corollary}[theorem]{Corollary}

\theoremstyle{definition}
\newtheorem{definition}[theorem]{Definition}
\newtheorem{example}[theorem]{Example}
\newtheorem{remark}[theorem]{Remark}

\newtheorem{construction}[theorem]{Construction}

\newcommand{\m}{\mathfrak}

\title{GARSIDE THEORY: A COMPOSITION--DIAMOND LEMMA POINT OF VIEW}

\author{Viktor Lopatkin}
\address{National Research University Higher School of Economics, Faculty of Computer Science,
Pokrovsky Boulevard 11, Moscow, 109028 Russia}
\email{wickktor@gmail.com}

\begin{document}

\maketitle

\begin{abstract} 
This paper shows how to obtain the key concepts and notations of Garside theory by using machinery of a Composition--Diamond lemma. We also show in some cases the greedy normal form is exactly a Gr\"obner--Shirshov normal form and a family of a left-cancellative category is to be a Garside family if and only if a suitable set of reductions is to be confluent up to some congruence on words.
\end{abstract}

\section*{Introduction}
Throughout this note, all rings are nonzero, associative with identity, all categories are small (objects and morphisms are sets), unless otherwise stated.

The Braid Theory has been formalized in the classical E. Artin's work \cite{Artin}. Later, in Garside's works it has been showed that the braid groups has a positive representation. This concept has been considered by S.I. Adjan \cite{Ad84} and W. Thurston\cite{Thur} (independently), as a result we have the Adjan--Thurston representation of the the braid group.

W. Thurston \cite{Thur} constructed a finite state automaton (see also \cite[Chapter 9]{EpThur}), having as the set of states the positive non-repeating braids, \textit{i.e.,} any two of its strands cross at most once. The concept of non-repeating braids is very useful, because the total (algebraic) number of crossings of two given strands in a braid is clearly an invariant of isotopy. Since a positive braid has only positive crossing, the absolute number of crossings of two strands in a positive braid is an invariant of isotopy. This idea is very useful, because we can forget about isotopic equivalence and, moreover, there is a bijection between the set of non-repeating braids and permutations.

An interesting characteristic of Thurston's automaton is that, after a word is imputed, the state is the maximal head of the word that lies in the set of non-repeating braids. This automaton allowed proving that the braid group is automatic. Moreover, this automaton rewrites any word into a canonical form which is called the (left or right) greedy normal form.

Roughly speaking, the describing of rewriting procedure of braids can be described as follows. We have a binary operation. This operation is resulted from the construction of this automaton, that is, the Thurston automaton works in the following way. Suppose we have two non-repeating braids $a$ and $b$, and we want to rewrite the word $ab$. The braid $a$ looks for a new crossing of the braid $b$, and if the braid $b$ allows to take this crossing (\textit{i.e.,} if there is a presentation $b = b'b''$ such that $b'$ is a braid which exactly contains the needed crossing for the braid $a$: $ab'$ is still non-repeating braid), then the braid $a$ takes this crossing. So, the operation ``add the needed crossing'' from the braid $b$ to the braid $a$ can be interpreted as a ``head'' function of the braid $ab$, denoted by $H(ab)$, \textit{i.e.,} we can say that, the braid $a$ is hungry and greedy for new crossing every time. Then the rewriting procedure of a word $ab$ can be written as follows $ab = H(ab)T(ab)$, where a function $T$ is also called the tail function \cite[p. 294]{Dehbook}.  

The natural generalization of these ideas to some other monoids and categories is called now Garside Theory which has been developed by P. Dehornoy, F. Digne, E. Godele, D. Krammer, and J. Michel \cite{Dehbook}. This theory deals with left-cancellative categories (as a special case, one object, this includes the case of monoids). We refer the reader to book \cite{Dehbook} for complete historical background and some details.

If we look at the standard presentation (=the Coxeter presentation) of the symmetric group $\mathfrak{S}_n$, which is generated by $s_1,\ldots,s_{n-1}$ with relations
\begin{enumerate}
    \item $s_i^2 = 1,$ for all $1\le i \le n-1$,
    \item $s_is_j = s_js_i$ for all $1\le i,j \le n-1$ such that $|i-j|>1$, 
    \item $s_{i+1}s_{i}s_{i+1} = s_is_{i+1}s_i$ for all $1\le i \le n-2$,
\end{enumerate}
and at the Artin's presentation of the braid group $B_n$, which is generated by $\sigma_1,\ldots, \sigma_{n-1}$ with relations
\begin{enumerate}
    \item $\sigma_i\sigma_j = \sigma_j\sigma_i$ for all $1\le i,j \le n-1$ such that $|i-j|>1$, 
    \item $\sigma_{i+1}\sigma_{i}\sigma_{i+1} = \sigma_i\sigma_{i+1}\sigma_i$ for all $1\le i \le n-2$,
\end{enumerate}
we can see a similarity between these presentations. One obvious invariant of an isotopy of a braid is the permutation it induces on the order of the strands: given a braid $b$, the strands define a map $P(b)$ from the top set of endpoints to the bottom set of enpoints, which we interpret as a permutation of $\{1,\ldots,n\}$. In this way we get a homomorphism (epimorphism) $P_n: B_n \to \mathfrak{S}_n$, $\sigma_i \mapsto s_i$ for $1\le i \le n-1$. The inverse map (=a section) can by constructed by the Garside's result, from which follows that the braid group can be presented as a monoid which is generated by a set of divisors (left and right) of the braid $\Delta_n$ (=the Garside braid which can be described physically by rotating the $n$ strands together $180^\circ$ clockwise) and by the element $\Delta_n^{-1}$. For any permutation $\pi \in \mathfrak{S}_n$, we correspond the set $R_\pi:=\left\{(i,j):i<j \,\&\, \pi(i) > \pi(j)\right\}$, thus we obtain Adjan--Thurston's generators (or divisors of $\Delta_n$).

However, this idea can be done via the following interesting way which has been suggested by L.A. Bokut' (see \cite[3.1.4]{BokSurv}). Using the Gr\"obner--Shirshov normal form $\mathrm{NF}$ for the elements of $\mathfrak{S}_n$ (see \cite{BS01,BBCM}), we consider a monoid $B_n'$ generated by the following set of generators $\cup_{\pi \in \mathfrak{S}_n}\{ E_{\mathrm{NF}(\pi)}\}$ with relations $E_{\mathrm{NF}(\pi)}E_{\mathrm{NF}(\tau)} = E_{\mathrm{NF}(\pi\tau)}$ whenever $\ell(\mathrm{NF}(\pi \tau)) = \ell(\mathrm{NF}(\pi)) + \ell(\mathrm{NF}(\tau))$, here $\ell(u)$ means the length of a word $u.$ It is easy to show that the monoid $B_n'$ is isomorphic to the positive braid monoid $B_n^+$. Next, as it was shown in \cite[3.1.4]{BokSurv} the elements $E_{\mathrm{NF}(\pi)}$ are exactly the Adjan--Thurston generators. Further it was also shown that a set of polynomials $S = \{ab - H(ab)T(ab)\}$, where $a,b$ run all over the simple braids, is a Gr\"obner--Shirshov basis relative to some monomial order, and then, by the Composition--Diamond lemma, the corresponding Gr\"obner--Shirshov normal form is exactly the greedy normal form. 

In \cite[VI]{Dehbook} the concept of a Garside germ was introduced. Further in the book this concept was applied to Garside families etc. In particular, using this concept, in \cite[A, VI, Example 2.72]{Dehbook} it was shown how to obtain the braid monoid (and greedy normal form) via the symmetric group. This approach is very 
similar to Bokut's approach. In particular the elements of $\pi,\tau \in \mathfrak {S}_n$ are called {\em tight}\index{Group!braid!tight elements} if $\ell(\pi \tau) = \ell(\pi) + \ell(\tau)$ (in a sense described there).

In this paper, using the concept of germ in the sense of \cite[A, VI]{Dehbook}, we generalize L.A. Bokut' ideas and W. Thurston approach for braid groups to consider Garsdie theory via a Composition--Diamond lemma. We consider two categories, say, $\mathscr{C}$ and $\mathscr{A}$ with a surjective functor $P:\mathscr{C} \to \mathscr{A}$. Assume furhter that $\mathscr{A}$ can be presented as follows $\mathscr{A} = \mathsf{Cat}\langle \Gamma \, | \, R\rangle$ (=a presentation of the $\mathscr{A}$). By the Composition--Diamond lemma for categories (see \cite{Be, BCL12}) we can calculate a normal form $\mathrm{NF}$ (=a Gr\"onber--Shirshov normal form) and hence a $\mathbb{K}$-basis $\mathfrak{B}$ for the algebra $\mathbb{K}\mathscr{A}$ (= a category algebra, see Definition \ref{category_algebra}). We then consider a section $E:\mathscr{A} \to \mathscr{C}$ for the functor $P$ and we wonder know whether $E(a) E(b) = E(\mathrm{NF}(ab))$ for $a,b \in \mathfrak{B}$. This leads us to the concept of a germ (see Construction \ref{the_main_construction}) in the sense of \cite[A, VI]{Dehbook} and we take an interest on whether this germ, denoted by $\Upsilon_E(\mathscr{C},\mathscr{A},P,\mathfrak{B})$, is a Garside germ. The main result of this paper is Theorem \ref{The_Main_Result} answers whether such germs are Garside germs in terms of the confluence (up to a congruence) of the corresponding reduction system $\overline{S}_\Upsilon(\mathfrak{B})$ (see Construction \ref{reduction_system}).

For these reasons we need a new version (see Theorem \ref{CD-lemma}) of the classical Composition--Diamond lemma for categories (or for monoids). We consider a rewriting system $S$ on a set of words (=morphisms of the corresponding free category) with a congruence which satisfies some suitable conditions ($S$-admissibility, see Definition \ref{S-admissible}) and then we require that this system of reductions is to be confluent up to this congruence. This allows us to deal with the fact that an element of a left-cancellative category with a Gariside family $S$, in general, may have more than one $S$-normal decompositions. On the other hand, according to \cite[A, III, Proposition 15]{Dehbook} (see Proposition \ref{deform_of_S-greedy} in this paper), any two $S$-normal decomposition of an element are deformations by invertible elements of one another. This is why we need a ``deformated'' version of a Composition--Diamond lemma.

In particular (see \ref{C=A}), when $\mathscr{C}=\mathscr{A} = \mathsf{Cat}\langle \Gamma \, | \, R \rangle$, and $P$ is assumed to be an identity functor $\mathbf{id}_\mathscr{C}$, a criteria for a germ $\Upsilon(\mathscr{A},\mathfrak{B}): = \Upsilon_E(\mathscr{C},\mathscr{C},\mathbf{id}_\mathscr{C},\mathfrak{B})$ to be a Garside germ is equivalent to a choosing a suitable subset of a set $\mathfrak{B}$ of basic elements of $\mathscr{C}$ (= irreducible elements). In other words to calculate a Garside family of a left-cancellative category $\mathscr{C}$ we first of all have to calculate a Gr\"obner--Shirshov basis for the ideal $I(R)$, this gives a basis $\mathrm{Irr}(R)$, and then we chose a subset of of this basis to make a Garside family. This can be done by considering the corresponding system of reductions (see Construction \ref{reduction_system}) and then we have to check that all ambiguity are resolvable. In the case when a category $\mathscr{C}$ has nontrivial invertible elements we have to require that all such ambiguities are to be resolvable up to congruence $\thickapprox$, where $\thickapprox$ is a deformation of elements of $\mathscr{C}$ by invertible elements (see Definition \ref{def_of_deformation}). We demonstrate how this works for the Klein bottle monoid (Example \ref{the_Klein_bottle_monoid}).

\section{Basic Notations of the Garside Theory}
In this section we recall the basic concept of the Garside theory. We refer to \cite{Dehbook} (see also the survey \cite{Dehsurv}) for details.

We deal with categories and correspondence terminology, let us remind some basic definitions and concepts.

\begin{definition}
A {\it precategory} $\mathscr{P}$ is a pair $(\mathrm{Ob}(\mathscr{P}),\mathrm{Hom}(\mathscr{P}))$ with two maps $\mathfrak{s},\mathfrak{t}:\mathrm{Hom}(\mathscr{P})\to \mathrm{Ob}(\mathscr{P})$ which are called {\it source} and {\it target}, respectively. The elements of $\mathrm{Ob}(\mathscr{P})$ are called the objects, those of $\mathrm{Hom}(\mathscr{P})$ are called the elements or morphisms. 
\end{definition}

By definition, a category is a precategory, plus composition map that obeys certain rules. Since for every object $x$ there is a morphism $1_x:x \to x$, we can consider only morphisms. Identity elements $1_x$ are also called \textit{trivial}, and the collection of all trivial elements in $\mathscr{C}$ is denoted by $\mathbf{1}_\mathscr{C}.$

Assume that $\mathscr{P}$ is a precategory. For $p\ge 1$ and $x,y \in \mathrm{\mathscr{P}}$, an \textit{$\mathscr{P}$-path} of length $p$ with source $x$ and target $y$ is a finite sequence $g_1,\ldots, g_p$ of elements of $\mathscr{P}$ such that the source of $g_1$ is $x$, the target of $g_p$ is $y$, and $g_1\cdots g_p$ is defined. The family of all $\mathscr{P}$-paths of length $p$ is denoted by $\mathscr{P}^{[p]}$.

\begin{definition}[{\cite[II, 7,8]{Mac}}]
Let ${\Gamma} = (V,E)$ be an oriented graph. We say that \textit{a graph $\Gamma = (V,E)$ generates a category} $\mathscr{C}$ if $\mathrm{Ob}(\mathscr{C}) = V$, and any morhpism of $\mathscr{C}$ will be the strings of compasable paths of $\Gamma$, so that a morphism of $\alpha:b \to a$ may be pictured as a path form $b$ to $a$, consisting of successive edges, say $e_1,\ldots, e_n$, of $\Gamma,$ we then say that the $\alpha$ has length $n$, and we write $\ell(\alpha) = n.$ This category $\mathscr{C}$ will be written $\mathsf{Cat}\langle \Gamma \rangle$ and called the \textit{free category} generated by the graph ${\Gamma}$.
\end{definition}

\begin{example}
Let $\Gamma$ have a single vertex, say, $v$, and $n\ge 1$ edges $e_1,\ldots, e_n$. Then $\mathsf{Cat}\langle \Gamma \rangle$ has the following semigroup presentation (= as a presentation of a semigroup)
 \[
  \mathsf{Cat}\langle \Gamma \rangle = \mathsf{Smg}\langle 1_v, e_1,\ldots, e_n\, | \, 1_v^2 =1_v, 1_ve_1 = e_11_v  = e_1,\ldots, 1_ve_n = e_n1_v = e_n \rangle, 
 \]
 it follows that we get the free monoid generated by $e_1,\ldots, e_n$ and where $1_v$ is its identity element (=empty word).
\end{example}

\begin{definition}
 Let $\mathscr{C}$ be a category. A function $\mathcal{R}$ which assigns to each par of objects $a,b$ of $\mathscr{C}$ a binary relation $\mathcal{R}_{a,b}$ on the set $\mathscr{C}(a,b)$ (=all morphism from $a$ to $b$) is called a \textit{congruence} on $\mathsf{Cat}\langle \Gamma  \rangle$ \textit{i.e.,}
 \begin{enumerate}
     \item for each pair $x,y$ of objects, $\mathcal{R}_{a,b}$ is a reflexive, symmetric, and transitive relation on $\mathsf{Cat}\langle \Gamma \rangle (x,y)$.
     \item if $f,f': x \to y$ have $f \mathcal{R}_{x,y} f'$, then for all $g:x'\to x$ and all $h:y \to y'$ one has $(hfg) \mathcal{R}_{x',y'} (hf'g).$
 \end{enumerate}
 
 In case $\mathsf{Cat}\langle \Gamma \rangle$ is the free category generated by a graph $\Gamma$ we shall call $\mathsf{Cat}\langle \Gamma \, | \, \mathcal{R} \rangle:=\mathsf{Cat}\langle \Gamma \rangle/\mathcal{R}$ the category with generators $\Gamma$ and relations $\mathcal{R}.$
 \begin{flushright}
 $\square$
 \end{flushright}
\end{definition}

As a special case (one object), this includes the case of a monoid given by generators and relations.

Throughout this note we frequently use the following notations; $\mathsf{Cat}\langle \Gamma \, | \, \mathcal{R} \rangle$ (\textit{resp.} $\mathsf{Smg}\langle \Gamma \, | \, \mathcal{R} \rangle$, \textit{resp.} $\mathsf{Gr}\langle \Gamma \, | \, \mathcal{R} \rangle$) means a category (\textit{resp.} semigroup, \textit{resp.} group) generated by a graph (\textit{resp.} a set) $\Gamma$ and relations $\mathcal{R}$.

Let $\mathscr{C}$ be a category. Every subset $\mathscr{S} \subseteq \mathscr{C}$ is called subfamily, \textit{i.e.,} $\mathscr{S}$ is the precategory made of $\mathscr{S}$ together with the restriction of the source and target maps to $\mathscr{S}$.

\begin{definition}{\cite[A, III, Definitions 1.1, 1.17]{Dehbook}}~

A category $\mathscr{C}$ is called {\it left-cancellative} if every relation $fg = fh$ with $f,g,h\in\mathscr{C}$ implies $g=h$.

A subfamily $\mathscr{S}$ of a left-cancellative category $\mathscr{C}$ is to be {\it closed under right-divisor} if every right-divisor of an element $\mathscr{S}$ is an element of $\mathscr{S}$.

A length-two $\mathscr{C}$-path $b\cdot c \in \mathscr{C}^{[2]}$ is called \textit{$\mathscr{S}$-greedy} if each relation $s \preccurlyeq a bc$ with $s$ in $\mathscr{S}$ implies $s \preccurlyeq ab$.

A path $a_1\cdots a_n$ is called \textit{$\mathscr{S}$-greedy} if $a_i\cdot a_{i+1}$ is $\mathscr{S}$-greedy for each $i<n.$
 
The family of all invertible elements in a category $\mathscr{C}$ is denoted by $\mathscr{C}^\times$. Let $\mathscr{S} \subseteq \mathscr{C}$ be a family in $\mathscr{C}$, we set
\[
 \mathscr{S}^\sharp:=\mathscr{S}\mathscr{C}^\times \cup \mathscr{C}^\times,
\]
here $\mathscr{S}\mathscr{C}^\times$ consists of all elements $f\in \mathscr{C}$ of the form $f = sc$ where $s\in \mathscr{S}$ and $c \in \mathscr{C}^\times$.

A path is \textit{$\mathscr{S}$-normal} if it is $\mathscr{S}$-greedy and its entries lie in $\mathscr{S}^\sharp$.

A procedure of finding $S$-normal form is called {\it Garside normalization}.
\end{definition}

\begin{example}Let $W(X)$ be a free monoid generated by nonempty set $X$. Let $S$ be the family of all squarefree words in $W(X)$, \textit{i.e.,} the words that admit no factors of the form $x^2$. It is obviously that every word $w\in W(X)$ admits a longest prefix that is squarefree.
 
 For instance, let $X = \{a,b,c\}$, and take $w = ab^2cabc^2babca$, then its $S$-normalization is
 \[
  w= ab\cdot bcac \cdot cbabca.
 \]
\end{example}

\begin{example}[\textbf{Free Abelian Monoids \cite[A, I.1.1]{Dehbook}}]
Take $n \ge 1$, and consider the free abelian monoid $\mathbb{N}^n$. It is clear that any its element $g$ can be viewed as a map $g: \{1,\ldots, n\} \to \mathbb{N}$. Denote by $g(k)$ the $k$th entry of $g.$

For $g, g' \in \mathbb{N}$, we define $g\cdot g'(k): = g(k) + g'(k)$ for each $k$.

Set $\Delta_n$ by $\Delta_n(k) = 1$ for every $k$, and put $N_n:=\{\eta \in \mathbb{N}^n\, | \, \eta(k) \in \{0,1\} \mbox{ for any $k$}\}$, and, finally, for $f,g\in \mathbb{N}^n$, say that $f \le g$ is true if $f(i) \le g(i)$ holds for every $i \in \{1,\ldots, n\}$.

\begin{proposition}[{\cite[A, I, Proposition 1.1]{Dehbook}}]
Every element of $\mathbb{N}^n$ admits a unique decomposition of the form $\Delta^d_n\eta_1\cdots \eta_p$ with $d \in \mathbb{N}$ and $\eta_1,\ldots, \eta_p \in N_n$ satisfying $\eta_1 \ne \Delta_n$, $\eta_p \ne 1$, and, for every $i<p$, and for every $g
\in \mathbb{N}^n\setminus\{1\}$ we have $g \le \eta_{i+1}$ implies $\eta_i g \not\le \Delta_n.$
\end{proposition}

For instance, let $n=3$, and take $f = (5,4,3)$,  then its $N_n$-normalization is
\begin{eqnarray*}
 (5,4,3) &=& (3,3,3) + (2,1,0) \\
 &=& (3,3,3) + (1,1,0) + (1,0,0),
\end{eqnarray*}
thus we obtain $f = \Delta_3^3\cdot \eta_1 \cdot \eta_2$
\end{example}

\begin{definition}[{\cite[A, III, Definition 1.20]{Dehbook}}]\label{def_of_deformation}
 (See Fig.\ref{deformofpaths}) Assume that $\mathscr{C}$ is a left-cancellative category. A $\mathscr{C}$-path $a =a_1\cdots a_n$ is said to be a \textit{deformation by invertible elements} or \textit{$\mathscr{C}^\times$-deformation}, of another $\mathscr{C}$-path $b=b_1\cdots b_m$, we then write $a\thickapprox b$, if there exist $\epsilon_0,\ldots, \epsilon_\ell \in \mathscr{C}^\times$, $\ell = \mathrm{max}(n,m)$, such that $\epsilon_0$ and $\epsilon_\ell$ are identity elements and $\epsilon_{i-1}b_i = a_i\epsilon_i$ holds for $1\le i \le \ell$, where, for $n \ne m$, the shorter path is expanded by identity elements.
\end{definition}

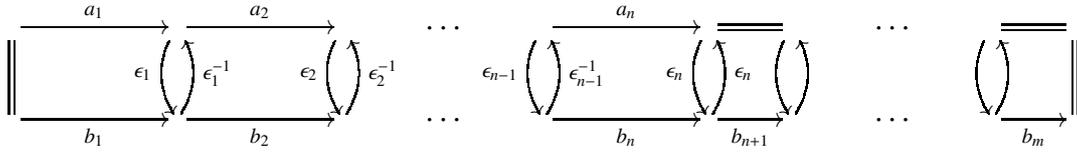
\begin{figure}[h!]
    \centering
     \[
      \xymatrix{
       \ar@{=}[d] \ar@{->}[rr]^{a_1} && \ar@/_/[d]_{\epsilon_1} \ar@{->}[rr]^{a_2} && \ar@/_/[d]_{\epsilon_2} & \cdots & \ar@/_/[d]_{\epsilon_{n-1}} \ar@{->}[rr]^{a_n} && \ar@/_/[d]_{\epsilon_n} \ar@{=}[r] &  \ar@/_/[d] & \cdots & \ar@/_/[d] \ar@{=}[r] & \ar@{=}[d]\\
       \ar@{->}[rr]_{b_1}  && \ar@/_/[u]_{\epsilon_1^{-1}} \ar@{->}[rr]_{b_2}  && \ar@/_/[u]_{\epsilon_2^{-1}} & \cdots & \ar@/_/[u]_{\epsilon_{n-1}^{-1}} \ar@{->}[rr]_{b_n}&& \ar@/_/[u]_{\epsilon_n} \ar@{->}[r]_{b_{n+1}} & \ar@/_/[u] & \cdots & \ar@/_/[u] \ar@{->}[r]_{b_m}&
      }
     \]
    \caption{Deformation by invertible elements: invertible elements connect the corresponding entries; if one path is shorter (here we are in the case $n<m$), it is extended by identity elements.}
    \label{deformofpaths}
\end{figure}

\begin{proposition}[{\cite[A, III, Proposition 1.25]{Dehbook}}]\label{deform_of_S-greedy}
  If $S$ is a subfaily of a left-cancellative category $\mathscr{C}$, any two $S$-normal decomposition of an element of $\mathscr{C}$ are $\mathscr{C}^\times$-deformations of one another.
\end{proposition}

\begin{example}
Let us consider a symmetric group $\mathfrak{S}_n$, $n \ge 3$, in Coxeter presentation:
  \[
 \mathfrak{S}_n = \mathsf{Smg}\langle s_1,\ldots, s_{n-1}\ | \, \mathcal{R} \rangle,
\]
where $\mathcal{R}$ is the following set of relations:
\begin{enumerate}
    \item $s_i^2 = 1,$ for all $1\le i \le n-1$,
    \item $s_is_j = s_js_i$ for all $1\le i,j \le n-1$ such that $|i-j|>1$, 
    \item $s_{i+1}s_{i}s_{i+1} = s_is_{i+1}s_i$ for all $1\le i \le n-2$.
\end{enumerate}

We have $s_is_j \thickapprox s_js_i$ for any $1\le i,j\le n-1$, $|i-j|>1$ because the diagram
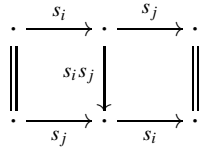
\[
 \xymatrix{
  \cdot \ar@{->}[r]^{s_i} \ar@{=}[d] & \cdot \ar@{->}[d]_{s_is_j} \ar@{->}[r]^{s_j} & \cdot \ar@{=}[d] \\
  \cdot \ar@{->}[r]_{s_j}  & \cdot \ar@{->}[r]_{s_i} & \cdot
 }
\]
is commutative, and $s_ks_{k+1}s_k \thickapprox s_{k+1}s_ks_{k+1}$ for all $1\le k\le n-1$ because the diagram 
\[
 \xymatrix{
  \cdot \ar@{->}[r]^{s_k} \ar@{=}[d] & \cdot \ar@{->}[r]^{s_{k+1}} \ar@{->}[d]^{s_ks_{k+1}} & \cdot \ar@{->}[r]^{s_k} \ar@{->}[d]^{s_ks_{k+1}} & \cdot \ar@{=}[d] \\
  \cdot \ar@{->}[r]_{s_{k+1}} & \cdot \ar@{->}[r]_{s_{k}} & \cdot \ar@{->}[r]_{s_{k+1}} & \cdot
 }
\]
is commutative. \begin{flushright}
 $\square$
\end{flushright}
\end{example}

\begin{definition}{\cite[A, III, Definition 1.31]{Dehbook}}
 A subfamily $\mathscr{S}$ of a left-cancellative category $\mathscr{C}$ is called a \textit{Garside} family in $\mathscr{C}$ if every element of $\mathscr{C}$ admits at least one $\mathscr{S}$-normal decomposition.
 
 A subfamily $\mathscr{S}$ of a category $\mathscr{C}$ is said to be {\it solid} in $\mathscr{C}$ if $\mathscr{S}$ includes $\mathbf{1}_\mathscr{C}$ and it is closed under right-divisor.
\end{definition}

\begin{definition}{\cite[A, VI, Definitions 1.3, 1.23]{Dehbook}}\label{def_of_germ}\\
  A germ is a triple $\bigl(\Upsilon,\bullet,\mathbf{1}_\Upsilon\bigr)$ where $\Upsilon$ is a precategory, $\mathbf{1}_\Upsilon$ is a subfamily of an elements $x$ with source and target $x$ for each object $x$, and the partial map $\bullet$ that satisfies
  \begin{itemize}
    \item[(1)] if $\alpha \bullet \beta$ is defined, the source of $\alpha \bullet \beta$ is the source of $\alpha$, and its target is the target of $\beta$;
    \item[(2)] the relation $1_x \bullet \alpha = \alpha = \alpha \bullet 1_y$ hold for each $\alpha$ in $\Upsilon(x,y)$;
    \item[(3)] if $\alpha \bullet \beta$ and $\beta\bullet \gamma$ are defined, then $(\alpha \bullet \beta)\bullet \gamma$ is defined if and only if $\alpha\bullet (\beta\bullet \gamma)$ is, in which case they are equal.
    \end{itemize}
  The germ is said to be left-associative (\textit{resp.} right-associative) if
  \begin{itemize}
    \item[(4)] $(\alpha\bullet \beta)\bullet \gamma$ is defined, then so is $\beta\bullet \gamma$;
    \item[(5)] (\textit{resp.} $\alpha \bullet (\beta\bullet \gamma)$ is defined, then so is $\alpha \bullet \beta$).
  \end{itemize}

If $(\Upsilon,\bullet,\mathbf{1}_\Upsilon)$ is a germ, we denote by $\mathsf{Cat}\langle \Upsilon,\bullet,\mathbf{1}_\Upsilon\rangle$ (or shortly $\mathsf{Cat}\langle\Upsilon\rangle$) the category $\mathsf{Cat}\langle\Upsilon\,|\,\mathcal{R}_\bullet \rangle$, where $\mathcal{R}_\bullet$ is the family of all relations $\alpha\bullet \beta = \alpha\beta$ with $\alpha,\beta\in \Upsilon$ and $\alpha\bullet \beta$ is defined.
  
A germ $\Upsilon$ is called \textit{a Garside germ} if it is a Garside family for the category $\mathsf{Cat}\langle\Upsilon\rangle$.
\end{definition}

\begin{example}
Let us consider the positive braid monoid in three strands \[
B_3^+ = \mathsf{Smg} \langle \sigma_1,\sigma_2,\, | \, \sigma_1\sigma_2\sigma_1 = \sigma_2\sigma_1\sigma_2 \rangle,
\]
and let $\Upsilon$ be the following family in $B_3^+$, $\Upsilon = \{1,\sigma_1,\sigma_2,\sigma_1\sigma_2,\sigma_2\sigma_1, \Delta_3\},$ where $\Delta_3:=\sigma_1\sigma_2\sigma_1 =\sigma_2\sigma_1\sigma_2.$

Define the partial map $\bullet:\Upsilon^{[2]} \to \Upsilon$ as follows

\begin{center}
 \begin{tabular}{c||c|c|c|c|c|c}
 $\bullet$ & $1$ & $\sigma_1$ & $\sigma_2$ & $\sigma_1\sigma_2$ & $\sigma_2\sigma_1$ & $\Delta_3$\\
 \hline
 \hline
 $1$ & $1$ & $\sigma_1$ & $\sigma_2$ & $\sigma_1\sigma_2$ & $\sigma_2\sigma_1$ & $\Delta_3$\\ 
 \hline
 $\sigma_1$ & $\sigma_1$ &  & $\sigma_1\sigma_2$ &  & $\Delta_3$ & \\ 
 \hline
 $\sigma_2$ & $\sigma_2$ & $\sigma_2\sigma_1$ &  & $\Delta_3$ &  & \\ 
 \hline
 $\sigma_1\sigma_2$ & $\sigma_1\sigma_2$ & $\Delta_3$ &  &  &  & \\ 
 \hline
 $\sigma_2\sigma_1$ & $\sigma_2\sigma_1$ &  & $\Delta_3$ &  &  & \\ 
 \hline
 $\Delta_3$ & $\Delta_3$ &  &  &  &  & \\ 
 \hline
 \end{tabular}
\end{center}

By the straightforward verification, it is easy to see that $(\Upsilon,\bullet)$ is a left-cancellative and left-associative germ.
\end{example}

\begin{definition}[{\cite[A, VI, Definitions 2.1, 2.27]{Dehbook}}]\label{def_of_I&J}
  For $\Upsilon$ a germ and $\alpha,\beta \in\Upsilon$, we put
  \begin{align*}
    &\mathcal{I}_\Upsilon(\alpha,\beta):=\bigl\{\delta \in \Upsilon\, |\, \exists \gamma \in \Upsilon \, (\delta = \alpha \bullet \gamma \mbox{ and } \gamma \preccurlyeq_\Upsilon \beta)\bigr\},\\
    &\mathcal{J}_\Upsilon(\alpha,\beta): = \bigl\{\gamma \in \Upsilon\, |\, \beta \bullet \gamma \in \Upsilon \, \&\, \gamma \preccurlyeq_\Upsilon \beta\bigr\}.
  \end{align*}
  
  A map $\mathfrak{I}: \Upsilon^{[2]}\to \Upsilon$ is called an \textit{$\mathcal{I}$-function} (\textit{resp.} a \textit{$\mathcal{J}$-function}) if, for every $\alpha \cdot \beta \in \Upsilon^{[2]}$, the value at $\alpha \cdot \beta$ lies in $\mathcal{I}_\Upsilon(\alpha,\beta)$ (\textit{resp.} in $\mathcal{J}_\Upsilon(\alpha, \beta)$.)

  An $\mathcal{I}$-function $\mathfrak{I}:\Upsilon^{[2]} \to \Upsilon$ is called greatest $\mathcal{I}$-function if and only if for every $\alpha \cdot \beta \in \Upsilon^{[2]}$, we have $\mathfrak{I}(\alpha,\beta) = \alpha \bullet \gamma$ for some $\gamma \in \Upsilon$ satisfying $\gamma \preccurlyeq_\Upsilon \beta$, and $\delta \preccurlyeq_\Upsilon \mathfrak{I}(\alpha,\beta)$ holds for every $\delta \in \mathcal{I}_\Upsilon(\alpha,\beta)$.
\end{definition}

\begin{proposition}[{\cite[A, VI, Proposition 2.28]{Dehbook}}]\label{prop VI,2.28}
  A germ $\Upsilon$ is a Garside germ if and only if, $\Upsilon$ is left-associative, left-cancellative, and admits a greatest $\mathcal{I}$-function.
\end{proposition}

\begin{definition}[\textbf{$\mathcal{I}$-law}]{\cite[A, VI, Definition 2.5]{Dehbook}}
If $\Upsilon$ is a germ and $\mathfrak{I}$ is a map from $\Upsilon^{[2]}$ to $\Upsilon$, we say that $\mathfrak{I}$ obeys $\mathcal{I}$-law if, for every $\alpha,\beta,\gamma \in \Upsilon$ with $\alpha \bullet \beta$ defined, we have
  \begin{equation}\label{I-law}
      \mathfrak{I}(\alpha, \mathfrak{I}(\beta,\gamma)) =^\times \mathfrak{I}(\alpha \bullet \beta, \gamma).
  \end{equation}
  If the counterpart of this equality with $=$ replacing $=^\times$, we say that $I$ obeys the sharp  $\mathcal{I}$-law.
\end{definition}

\begin{theorem}[{\cite[A, VI, Proposition 2.8]{Dehbook}}]\label{RecGar_germ}
  A germ $\Upsilon$ is a Garside germ if and only if it is left-associative, left-cancellative, and admits and $\mathcal{I}$-function obeying the sharp $\mathcal{I}$-law.
\end{theorem}

\begin{proposition}[{\cite[A, VI, Lemma 2.4]{Dehbook}}]\label{greedy}
  If $\Upsilon$ is a Garside germ, the following are equivalent for every $\alpha, \beta \in \Upsilon$;
  \begin{enumerate}
      \item The path $\alpha \cdot \beta$ is $\Upsilon$-normal.
      \item The element $\alpha$ is $\prec_\Upsilon$-maximal in $\mathcal{I}_\Upsilon(\alpha,\beta)$.
      \item Every element of $\mathcal{J}_\Upsilon(\alpha,\beta)$ is invertible in $\Upsilon.$
  \end{enumerate}
\end{proposition}

\begin{proposition}\label{Garside_germ=cat}
 Let $\mathscr{S}$ be a solid family generating a left-cancellative category $\mathscr{C}$, then $\mathscr{S}$ equipped with the induced partial product is a germ $\Upsilon(\mathscr{S})$. If the $\Upsilon(\mathscr{S})$ is a Garside germ then $\mathscr{C}$ is isomorphic to $\mathsf{Cat}\langle\Upsilon(\mathscr{S}) \rangle$.
\end{proposition}
\begin{proof}
To prove this we need the following results
\begin{lemma}[{\cite[A, VI, Proposition 1.1]{Dehbook}}]
   If $\mathscr{S}$ is a solid Garside family in a left-cacellative category $\mathscr{C}$, then $\mathscr{S}$ equipped with the induced partial product is a germ $\Upsilon(\mathscr{S})$ and $\mathscr{C}$ is isomorphic to $\mathsf{Cat}\langle\Upsilon(\mathscr{S}) \rangle$,
\end{lemma}
\begin{lemma}[{ \cite[A, IV, Proposition 2.24]{Dehbook}}]
 A solid generating subfamily $\mathscr{S}$ of a left-cancellative category $\mathscr{C}$ is a Garside family if and only if there exists $\mathfrak{I}:\mathscr{S}^{[2]} \to \mathscr{S}$ satisfying $\alpha \preccurlyeq \mathfrak{I}(\alpha,\beta) \preccurlyeq \alpha \cdot \beta$ for all $\alpha, \beta$ and $\mathfrak{I}(\alpha, \mathfrak{I}(\beta,\gamma)) = \mathfrak{I}(\alpha\beta, \gamma)$ for every $\alpha,\beta,\gamma \in \mathscr{S}$ satisfying $\alpha \cdot \beta \in \mathscr{S}.$   
\end{lemma}

Since $\Upsilon(\mathscr{S})$ is assumed to be a Garside germ then there is a $\mathcal{I}$-greatest function $\mathfrak{I}$. It is clear that it satisfies the statement above (see also proof of \cite[A, IV, Lemma 2.23, Proposition 2.24]{Dehbook} and \cite[IV, Proposition 2.8]{Dehbook}). Hence, by \cite[A, IV, Proposition 2.24]{Dehbook}, $\mathscr{S}$ is a Garside family in $\mathscr{C}$. Therefore, by Proposition \cite[A, VI, Proposition 1.1]{Dehbook}, the statement follows.
\end{proof}

\begin{theorem}[\textbf{germ from  Garside},{\cite[A, VI, Proposition 1.1]{Dehbook}}]\label{germ_from_Garside}
  If $\mathscr{S}$ is a solid Garside family in a left-cacellative category $\mathscr{C}$, then $\mathscr{S}$ equipped with the induced partial product is a germ $\Upsilon(\mathscr{S})$ and $\mathscr{C}$ is isomorphic to $\mathsf{Cat}\langle\Upsilon(\mathscr{S}) \rangle$.
\end{theorem}

\begin{remark}
 If a solid family $\mathscr{S}$ in a left-cancellative category $\mathscr{C}$ generates $\mathscr{C}$ and $\Upsilon(\mathscr{S})$ is a Garside germ then $\mathscr{C} \cong \mathsf{Cat}\langle \Upsilon(\mathscr{S})\rangle$ 
\end{remark}

\section{A Composition--Diamond Lemma for Categories}
In this section we present a ``defformed'' version of a Composition--Diamond lemma for categories. We essentially follows \cite{Be}. A sketch of a Composition--Diamond lemma for categories has been appeared in \cite[9.3]{Be} and it has been developed in detail in \cite{BCL12} (see also \cite[2.4]{BokSurv}).

\begin{definition}\label{category_algebra}
 Let $\mathbb{K}$ be a commutative ring and $\mathscr{C}$ a category. Define the \textit{category algebra} $\mathbb{K}\mathscr{C}$ to be the the free $\mathbb{K}$-module with the morphisms of the category $\mathscr{C}$ as a basis. The product of morphisms $a$ and $b$ as elements of $\mathbb{K}\mathscr{C}$ is defined to be
 \[
  a\cdot b: = \begin{cases} a b & \mbox{if $a$ and $b$ can be composed,} \\ 0 & \mbox{otherwise,} \end{cases}
 \]
 and this product is extended to the whole of $\mathbb{K}\mathscr{C}$ using bilinearity of multiplication.
\end{definition}

Given a graph $\Gamma$, consider the corresponding free category $\mathsf{Cat}\langle \Gamma \rangle$. Let $S$ be a set of pairs of the form $\sigma = (W_\sigma, \omega_\sigma)$, where $W_\sigma, \omega_\sigma \in \mathsf{Cat}\langle \Gamma \rangle$. For any $\sigma \in S$ and $A,B \in \mathsf{Cat}\langle \Gamma \rangle$, let $\mathfrak{r}_{A\sigma B}: \mathsf{Cat}\langle \Gamma \rangle \to \mathsf{Cat}\langle \Gamma \rangle$ be a functor defined as follows:
\[
 \mathfrak{r}_{A\sigma B}(W): = \begin{cases}
  A\omega_\sigma B & \mbox{if $W = W_\sigma$} \\
  W & \mbox{if $W \ne W_\sigma.$}
 \end{cases}
\]

We call the given set $S$ a \textit{reduction system}, and the maps $\mathfrak{r}_{A\sigma B}$ \textit{reductions,} if $\mathfrak{r}_{A\sigma B}(W) = W$, we say a reduction $\mathfrak{r}_{A\sigma B}$ acts \textit{trivially} on $W$, and we shall call $W$ \textit{irreducible under $S$} if every reduction is trivial on $W$, and we say that a word $W$ is \textit{reducible} in otherwise.

We shall frequently write $\mathfrak{r}_{A\sigma B}:W \to W'$ if $\mathfrak{r}_{A\sigma B}(W) = W'$.

Denote by $\mathrm{Irr}(S)$ the subset of all irreducible elements of $\mathsf{Cat}\langle \Gamma \rangle$ under all reductions from the set $S$. A finite sequence of reductions $\mathfrak{r}_1,\ldots, \mathfrak{r}_n$ will be said to be \textit{final} on $W\in \mathsf{Cat}\langle \Gamma \rangle$ if $\mathfrak{r}_n \cdots \mathfrak{r}_1(W) \in \mathrm{Irr}(S).$

An element $W$ of $\mathsf{Cat} \langle \Gamma \rangle$ will be called \textit{reduction-finite} if for every infinite sequence $\mathfrak{r}_1, \mathfrak{r}_2,\ldots,$ of reductions, $\mathfrak{r}_i$ acts trivially on $\mathfrak{r}_{i-1}\cdots \mathfrak{r}_1(W)$ for all sufficiently large $i$.

\begin{definition}
 A 5-tuple $(\sigma,\tau, A,B,C)$ with $\sigma, \tau \in S$, and $A,B,C \in \mathsf{Cat}\langle \Gamma \rangle \setminus \{\mathbf{1}\}$, is called \textit{overlap} (\textit{resp.} \textit{inclusion}) \textit{ambiguity} of $S$ if $W_\sigma = AB$, $W_\tau = BC$ (\textit{resp.} if $\sigma \ne \tau$ and $W_\sigma = B$, $W_\tau = ABC$).
\end{definition}

\begin{definition}\label{S-admissible}
Let $\thickapprox$ be a congruence on a free category $\mathsf{Cat}\langle \Gamma \rangle$ and $S$ a set of reductions. We call a congruence $\thickapprox$ \textit{$S$-admissible} if the following hold:
\begin{enumerate}
    \item if $W \notin \mathrm{Irr}(S)$ then $W' \notin \mathrm{Irr}(S)$ for any $W' \thickapprox W$,
    \item if $\omega \in \mathrm{Irr}(S)$ then $\omega' \in \mathrm{Irr}(S)$ for any $\omega' \thickapprox \omega.$
\end{enumerate}
\end{definition}

An element $W \in \mathsf{Cat} \langle \Gamma \rangle$ is called \textit{reduction-unique up to $\thickapprox$ under $S$} if it is reduction-finite, and if its images under all final sequences reductions and under using an $S$-admissible congruence are the same. This common value will be denoted by ${\mathfrak{r}}_S(W),$ and we also write $\mathfrak{r}: W \to \mathfrak{r}_S(W)$ if $\mathfrak{r}$ is a composition of reductions of $S$ only and we write $\mathfrak{r}:W \rightsquigarrow \mathfrak{r}_S(W)$ in otherwise.

\begin{remark}
 Note that if all elements of $\mathsf{Cat}\langle \Gamma \rangle$ are reduction unique up to $\thickapprox$ then the following condition holds; if $W \rightsquigarrow U$, $W' \rightsquigarrow U'$ then from $W\thickapprox W'$ if follows that $U' \thickapprox U'.$
\end{remark}

\begin{definition}\label{addmisible_preorder}
 By a \textit{$\thickapprox$-admissible categorical partial preorder} on a free category $\mathsf{Cat}\langle \Gamma \rangle$ we shall mean a partial preorder $\lessapprox$ such that for $A,B,B',C \in \mathsf{Cat}\langle \Gamma \rangle$ the following hold:
 \begin{enumerate}
     \item if $B < B'$ then $ABC < AB'C$,
     \item whenever $B<B'$, $B'<B$ it must have that $B\thickapprox B'$,
     \item if $B<A$ and $B \thickapprox C$ then $C <A,$
     \item if $A<B$, $B\thickapprox C$ then $A<C.$
 \end{enumerate}
 
 Next, we say that a $\thickapprox$-admissible categorical partial preoder is \textit{compatible} with $S$ if for all $(W_\sigma, \omega_\sigma) = \sigma \in S$, $\omega_\sigma < W_\sigma$, $W_\sigma \not < \omega_\sigma$. 
\end{definition}

In the case with a single object, which includes the case of a free monoid, we also call a \textit{$\thickapprox$-admissible categorical partial preorder} as a \textit{$\thickapprox$-admissible monomial partial preorder}.

For a given set of reductions $S=\cup_{\sigma}\{(W_\sigma,\omega_\sigma)\}$ on a free category $\mathsf{Cat}\langle \Gamma \rangle$ we consider the following set of polynomials $\cup_{\sigma\in S}\{W_\sigma - \omega_\sigma\}$, and we say that \textit{$(S)$ is the corresponding set of polynomials for a given set of reductions $S$.}

Next, let $\thickapprox$ be an $S$-admissible congruence on $\mathsf{Cat}\langle \Gamma \rangle$. Set $\widetilde{(S)}:=(S) \cup (S')$, where $(S') = \cup_{\sigma\in S}\{W_\sigma-W'_\sigma\, |\, W_\sigma \thickapprox W_\sigma',\, W_\sigma,W_\sigma' \notin \mathrm{Irr}(S)\}$ and consider a two-sided ideal, denoted by $I_\thickapprox(S)$, generated by $\widetilde{(S)}$.

\begin{remark}
 Note that an ideal $I_\thickapprox(S)$ cannot be described as a two-sided ideal generated by the following set of polynomials $\cup_{\sigma\in S}\{ W_\sigma - \omega'_\sigma\, |\, W_\sigma \rightsquigarrow \omega_\sigma' \}$. Indeed, let $\mathfrak{r}_1: W_\sigma \to \omega_\sigma$, and $\mathfrak{r}_s: W'_\sigma \to \omega'_\sigma$ be reductions of $S$. Then we have $\mathfrak{r}': W_\sigma \rightsquigarrow \omega_\sigma'$ but $\mathfrak{r}'\notin S$, $W_\sigma \notin \mathrm{Irr}(S)$, and $\omega_\sigma \in \mathrm{Irr}(S)$, hence $W_\sigma - \omega'_\sigma \notin I_\thickapprox(S)$.
\end{remark}

\begin{definition}
An overlap ambiguity $(\sigma,\tau, A,B,C)$ (\textit{resp.} inclusion ambiguity) is called \textit{resolvable up to} $\thickapprox$ if there exist compositions of reductions $\mathfrak{r}:\omega_\sigma C \rightsquigarrow \mathfrak{r}(\omega_\sigma C)$ and $\mathfrak{r}':A\omega_\tau \rightsquigarrow \mathfrak{r}'(A\tau_\sigma)$ (\textit{resp.} $\mathfrak{r}:\omega_\sigma \rightsquigarrow \mathfrak{r}(\omega_\sigma)$, $\mathfrak{r}':A\omega_\tau C \rightsquigarrow \mathfrak{r}'(A\omega_\tau C)$) such that ${\mathfrak{r}}(\omega_\sigma C) \thickapprox {\mathfrak{r}'}(A\omega_\tau)$ (\textit{resp.} ${\mathfrak{r}}(\omega_\sigma) \thickapprox {\mathfrak{r}'}(A\omega_\tau C)$).

If all ambiguities of $S$ are resolvable up to $\thickapprox$ under $S$ we then say that $S$ is \textit{confluent up to $\thickapprox$.}
\end{definition}

Let $\lessapprox$ be a categorical partial preorder on a free category $\mathsf{Cat}\langle \Gamma \rangle$ compatible with a reduction system $S$. Consider $f_\sigma= W_\sigma - \omega_\sigma, f_\tau = W_\tau - \omega_\tau \in (S)$. There are two kinds of compositions:
\begin{itemize}
    \item[(1)] if $ABC = W_\sigma C = AW_\tau$ with $\ell(W_\sigma) + \ell(W_\tau)> \ell(ABC)$, then the polynomial $(f_\sigma,f_\tau)_{ABC}:=f_\sigma C-Af_\tau$ is called the \textit{intersection composition} of $f_\sigma$ and $f_\tau$ with respect to $ABC$,
    \item[(2)] if $ABC= W_\sigma = AW_\tau C$, then the polynomial $(f_\sigma,f_\tau)_{ABC}:=f_\sigma-Af_\tau C$ is called the \textit{inclusion composition} of $f_\sigma$ and $f_\tau$ with respect to $ABC.$
\end{itemize}

Consider $g\in \mathbb{K}\mathsf{Cat}\langle \Gamma \rangle$. Set $g \thickapprox 0$ if $g$ can be presented as follows $g = \sum_{i=1}^n(U_i-V_i)$, where for any $1 \le i \le n$, $U_i,V_i \in \mathsf{Cat}\langle \Gamma \rangle$ and all $U_i \thickapprox V_i$. 
We say $g \in \mathbb{K}\mathsf{Cat}\langle \Gamma \rangle$ is \textit{trivial modulo $((S), W)$ up to $\thickapprox$}, denoted by $g \thickapprox 0 \, (\bmod{(S), W})$, here $W \in \mathsf{Cat}\langle \Gamma \rangle$, if $g$ can be presented as follows $g = \sum_{i=1}^n\sum_{\sigma_i \in S}U_if_{\sigma_i} V_i + g'$, where all $U_i,V_i \in \mathsf{Cat}\langle \Gamma \rangle$, $UW_\sigma V < W$, $g'\in \mathbb{K}\mathsf{Cat}\langle \Gamma \rangle$, $g'\thickapprox 0,$ and for all its monomials, say $G_j'$, we have $G_j'< W$.  

\begin{definition}
Let $S$ be a set of reductions on a free category $\mathsf{Cat}\langle \Gamma \rangle$, $\thickapprox$ an $S$-admissible congruence, and $\lessapprox$ a $\thickapprox$-admissible categorical partial preoder compatible with $S$. The corresponding set of polynomials $(S)$ is called a \textit{$\thickapprox$-closed Gr\"obner--Shirshov} basis in $\mathbb{K}\mathsf{Cat}\langle \Gamma \rangle$ with respect to $\lessapprox$ if every composition $(f_\sigma,f_\tau)_{ABC}$ of polynomials $f_\sigma, f_\tau \in (S)$ is trivial modulo $((S), ABC)$ up to $\thickapprox$. 
\end{definition}

To stress that a set $X \subseteq \mathrm{Cat}\langle\Gamma \rangle$ is considered as the set with a congruence $\thickapprox$ on $\mathrm{Cat}\langle \Gamma \rangle$ we write $(X,\thickapprox).$

\begin{theorem}[{\textbf{a $\thickapprox$-closed Composition--Diamond Lemma}}]\label{CD-lemma}
Let $\mathbb{K}$ be a commutative ring with unit. Let $\mathsf{Cat}\langle \Gamma \rangle$ be a free category, $S$ a reduction system, $(S)$ the corresponding set of polynomials, $\thickapprox$ an $S$-admissible congruence on $\mathsf{Cat}\langle \Gamma \rangle$, and $\lessapprox$ a $\thickapprox$-admissible categorical partial preordering compatible with $S$ having descending chain condition. Then the following conditions are equivalent:
  \begin{itemize}
      \item[(1)] All ambiguities of $S$ are $\thickapprox$-resolvable.
      \item[(2)] $(S)$ is a $\thickapprox$-closed Gr\"obner--Shirshov basis relative to $\lessapprox.$
      \item[(3)] All elements of $\mathbb{K}\mathsf{Cat}\langle \Gamma \rangle$ are reduction-unique up to $\thickapprox$ under $S$.
      \item[(4)] $(\mathrm{Irr}(S),\thickapprox)$ is a linear $\mathbb{K}$-basis of the algebra $\mathbb{K}\mathsf{Cat}\langle \Gamma \, | \widetilde{(S)} \rangle = \mathbb{K}\mathsf{Cat}\langle \Gamma \rangle / I_\thickapprox(S)$, where $\mathrm{Irr}(S):= \{W \in \mathsf{Cat}\langle \Gamma \rangle \, | \, W \ne UW_\sigma V , \sigma \in S, U,V \in \mathsf{Cat}\langle \Gamma \rangle \}.$
  \end{itemize}
\end{theorem}
\begin{proof}
 We easily see from our general hypothesis, by induction with respect to the partial preordering with descending chain condition $\lessapprox$, that every element of $\mathbb{K}\mathsf{Cat}\langle \Gamma \rangle$, is reduction-finite up to $\thickapprox.$
 
 If $(S) = \cup_{i \in I}\{W_i - \omega_i\}$ is a $\thickapprox$-closed Gr\"obner--Shirshov basis relative to $\lessapprox$ it then gives the following system of reduction $S = \cup_{i\in I}\{\sigma_i = (W_i, \omega_i)\}$ is compatible with $\lessapprox$.
 
 $(1) \Longleftrightarrow (2)$ Take $f_\sigma, f_\tau \in (S)$ and let $AW_\sigma = W_\tau C = ABC$ we then get
 \begin{eqnarray*}
  (f_\sigma, f_\tau)_{ABC} &=& A(W_\sigma - \omega_\sigma) - (W_\tau - \omega_\tau)C \\
  &=& -A\omega_\sigma + \omega_\tau C
 \end{eqnarray*}
 Let $\mathfrak{r}_S^{(1)}:A\omega_\sigma \rightsquigarrow A_1W_{\sigma_1}C_1$ and $\mathfrak{r}^{(1)'}_S:\omega_\tau C \rightsquigarrow A_1'W_{\tau_1}C_1'$, we then get
 \begin{eqnarray*}
 (f_\sigma, f_\tau)_{ABC} &=&  -A\omega_\sigma + \omega_\tau C \\
 &\thickapprox& -A_1(f_{\sigma_1} + \omega_{\sigma_1})C_1 + A_1'(f_{\tau_1} + \omega_{\tau_1})C_1' \\
 &=& -A_1f_{\sigma_1}C_1 + A_1'f_{\tau_1}C_1' - A_1\omega_{\sigma_1}C_1 + A_1'\omega_{\tau_1}C_1'.
 \end{eqnarray*}
 
Next, setting $\mathfrak{r}^{(2)}_S: A_1\omega_{\sigma_1}C_1 \rightsquigarrow A_2W_{\sigma_2}C_2$, $\mathfrak{r}^{(2)'}_S: A_1'\omega_{\tau_1}C_1'\rightsquigarrow A_2'W_{\tau_2}C_2'$ and so on, after some steps, we get
 \[
 (f_\sigma, f_\tau)_{ABC} \thickapprox -\sum_{i=1}^NA_if_{\sigma_i}C_i + \sum_{j=1}^M A_j'f_{\tau_j}C_j' - A_{N}\omega_{\sigma_{N}}C_{N} + A_{M}'\omega_{\tau_{M}}C_{M}'.
 \]

It is easy to see that from Definition \ref{addmisible_preorder} it follows that $A_iW_{\sigma_1}C_i, A_j'W_{\tau_1}C_j' < ABC$, for any $1\le i \le N$, $1 \le j \le M$.

If $(S)$ is a $\thickapprox$-completed Gr\"obner--Shirshov basis relative to $\lessapprox$ then $N,M < \infty$, and $A_{N}\omega_{\sigma_{N}}C_{N} \thickapprox A_{M}'\omega_{\tau_{M}}C_{M}'$. It follows that the ambiguity $(\sigma, \tau, AB,BC)$ is resolvable up to $\thickapprox$. Similarly one can easy get the converse statement. 
 
 $(2) \Longrightarrow (3)$. It is clear that it is sufficient to prove all elements $W \in \mathsf{Cat}\langle \Gamma \rangle$ reduction-unique up to $\thickapprox.$ We must show that given any two reductions $\mathfrak{r}_{L\sigma M'}$ and $\mathfrak{r}_{L'\tau M}$ each acting nontrivially on $W$ we shall have $\mathfrak{r}_\mathcal{S}(\mathfrak{r}_{L\sigma M'}(W)) \thickapprox \mathfrak{r}_\mathcal{S}(\mathfrak{r}_{L'\sigma M}(W))$. There are three cases, according to the relative locations of the subwords $W_\sigma$ and $W_\tau$ in the element $W$. We may assume without loss of generality that $\ell(L)\le \ell(L')$.
 
 \textit{Case 1.} The subwords $W_\sigma$ and $W_\tau$, overlap in $W$, but neither contains the other. Then $W = UABCV$, where $(\sigma, \tau, A,B,C)$ is an overlap ambiguity of $S.$ Let us consider the corresponding composition for the polynomials $f_\sigma, f_\tau \in (S)$, we have 
  \[
 (f_\sigma, f_\tau)_{ABC} \thickapprox -\sum_{i=1}^NA_if_{\sigma_i}C_i + \sum_{j=1}^M A_j'f_{\tau_j}C_j' - A_{N}\omega_{\sigma_{N}}C_{N} + A_{M}'\omega_{\tau_{M}}C_{M}'.
 \]
 
 Since $(S)$ has been assumed to be $\thickapprox$-completed Gr\"obner--Shirshov basis relative to $\lessapprox$ we then obtain, $N,M < \infty$, $A_iW_{\sigma_i}C_i, A_i'W_{\tau_i}C_i' < ABC$ for all $1\le i \le N$, $1 \le j \le M$, and $A_{N+1}\omega_{\sigma_{N+1}}C_{N+1} \thickapprox A_{M+1}'\omega_{\tau_{M+1}}C_{M+1}'$. Next, since $\thickapprox$ is assumed to be congruence on $\mathsf{Cat}\langle \Gamma \rangle$ then $UA_{N+1}\omega_{\sigma_{N+1}}C_{N+1}V \thickapprox UA_{M+1}'\omega_{\tau_{M+1}}C_{M+1}'V,$ as required.
 
 \textit{Case 2.} One of the subwords $W_\sigma$, $W_\tau$ of $W$ is contained in the other. Thus case is handled like the preceding, using the resolvability up to $\thickapprox$ of inclusion composition.
 
 \textit{Case 3.} Finally, $W_\sigma$, $W_\tau$ are disjoint subwords of $W$, but in this case the statement is obvious. 

 $(3) \Longleftrightarrow (4)$ Any reduction $\mathfrak{r}_S$ induces the corresponding $\mathbb{K}$-module homomorphism on $\mathbb{K}\mathsf{Cat}\langle \Gamma \rangle$, denote it also by $\mathfrak{r}_S.$ 
 
 Assuming (3). Take $f \in \mathbb{K}\mathsf{Cat}\langle \Gamma \rangle$, say, $f = \sum_{i=1}^M \kappa_i u_i$, where all $\kappa_i \in \mathbb{K}$, $u_i \in \mathsf{Cat}\langle \Gamma \rangle$, and $u_1>u_2>\cdots > u_M$. Since all reductions of $S$ are assumed to be unique up to $\thickapprox$ then for any $1\le i \le M$, $u_i \thickapprox A_1^{(i)}W_{\sigma_{t_1}}^{(i)}B_{1}^{(i)}$ for some $A_1^{(i)}, B_{1}^{(i)} \in \mathsf{Cat}\langle \Gamma \rangle.$ Set $\varphi_\sigma:=W_\sigma - \omega_\sigma$ for any $\sigma = (W_\sigma,\omega_\sigma)\in S$. We can write $W_\sigma = \varphi_\sigma + \omega_\sigma$, thus, by induction, we get
\[
 u_i \thickapprox \sum_{k_i=1}^{m_i}A_{k_i}^{(i)} \varphi^{(i)}_{\sigma_{t_{k_i}}}B_{k_i}^{(i)} + C_{k_{m_i}}^{(i)},
\]
where $C_{k_{m_i}}^{(i)} \in \mathrm{Irr}(S)$, and all $A_{k_i}^{(i)} \mathrm{LT}(\varphi^{(p)}_{\sigma_{t_{k_i}}})B_{k_i}^{(i)}, C_{k_{m_i}}^{(i)} \lessapprox u_p$, here $\mathrm{LT}(\varphi)$ denotes a leading term of $\varphi$ related to $\lessapprox$. It follows that (cf. \cite[2.1, Lemma 2]{BokSurv}) for any polynomial $f \in \mathbb{K}\mathsf{Cat}\langle \Gamma \rangle$ we have
\[
  f \thickapprox \sum_{p=1}^N \alpha_p U_p \varphi_{\sigma_p} U_p' + \sum_{q=1}^{N'} \beta_q V_q
\]
where all $\alpha_p,\beta_q \in \mathbb{K}$, $U_p,U_p' \in \mathsf{Cat}\langle \Gamma \rangle$, $V_q \in \mathrm{Irr}(S)$, and $U_p\mathrm{LT}(\varphi_p)U_p', V_q\lessapprox \mathrm{LT}(f)$.

Next, since $\thickapprox$ is assumed to be $S$-admissible then $\mathbb{K}\mathsf{Cat}\langle \Gamma \rangle \cong I_\thickapprox(S) \oplus (\mathrm{Irr}(S),\thickapprox)$ and the statement (4) follows.

Conversely, assume (4) and suppose $W \in \mathsf{Cat}\langle \Gamma \rangle$ can be reduced to either of $\omega, \omega' \in (\mathrm{Irr}(S),\thickapprox)$. Then $W-\omega, W - \omega' \in I_\thickapprox(S)$ and hence $\omega-\omega' \in I_\thickapprox(S)$. On the other hand, $\omega - \omega' \in (\mathrm{Irr}(S),\thickapprox)$, it follows that $\omega - \omega' \in (\mathrm{Irr}(S),\thickapprox) \cap I_\thickapprox(S) = \{0\}$, because of $\thickapprox$ is assumed to be $S$-admissible, and proving (3).
 
 Finally, the implication $(3) \Longrightarrow (1)$ is immediate. This completes the proof.
\end{proof}

\begin{remark}\label{rem_of_empty_GSB}
 So we see that for a given free category $\mathsf{Cat}\langle \Gamma \rangle$ we district two relations on it; reductions and congruence $\thickapprox$. We do not consider an ambiguity up to $\thickapprox$, \textit{i.e.,} let $S$ be a set of reductions, then we would define an ambiguity up to $\thickapprox$ as follows, say a $5$-tuple $(\sigma, \tau, A,B,C)$ is called overlap (\textit{resp.} inclusion) ambiguity if $W_\sigma \thickapprox AB$ and $w_\tau \thickapprox BC$ (\textit{resp.} $W_\sigma \thickapprox B$, $W_\tau \thickapprox ABC$). But it is easy to see that in this case we just add to a reduction system new reductions of form $W_\sigma \to AB$ and $W_\tau \to BC$. This is why we district this two relations.
\end{remark}

\begin{remark}\label{Buch-Shirsh}
It is well known that similar to the Buchberger algorithm of completion for commutative polynomials, there is a Shirshov algorithm for noncommutative polynomials. However, in contrast to the case of commutative algebras, the Shirshov algorithm  in general does not terminate in a finite number of steps (we cannot use Noetherianity in noncommutative case). In our case we of course can use the same way. If a subset $(S) \subseteq \mathbb{K}\mathsf{Cat}\langle \Gamma \rangle$ is not a $\thickapprox$-closed Gr\"obner--Shirshov basis then one can add to $(S)$ all nontrivial up to $\thickapprox$ compositions of polynomials from $(S)$. Continuing this process repeatedly we finally obtain a $\thickapprox$-closed Gr\"obner--Shirshov basis $\widehat{(S)}$ that contains $S$. As in the case of classical Buchbegrger--Shirshov algorithm we have to assume that the initial set $(S)$ is recursively enumerable (e.g., finite) and that the underlying ring $K$ is computable. The resulting set $\widehat{(S)}$ is also recursively enumerable but not necessarily finite even if $(S)$ was.
\end{remark}

\begin{example} Let us consider the symmetric group $\mathfrak{S}_n$, $n\ge 3$, in a Coxeter presentation
  \[
 \mathfrak{S}_n = \mathsf{Smg}\langle s_1,\ldots, s_{n-1}\ | \, \mathcal{R} \rangle,
\]
where $\mathcal{R}$ is the following set of relations:
\begin{enumerate}
    \item $s_i^2 = 1,$ for all $1\le i \le n-1$,
    \item $s_is_j = s_js_i$ for all $1\le i,j \le n-1$ such that $|i-j|>1$, 
    \item $s_{i+1}s_{i}s_{i+1} = s_is_{i+1}s_i$ for all $1\le i \le n-2$.
\end{enumerate}

We have already remarked that $s_ks_{k+1}s_k \thickapprox s_{k+1}s_ks_{k+1}$ and $s_is_j \thickapprox s_js_i$ because the corresponding diagrams
\[
 \xymatrix{
  \cdot \ar@{->}[r]^{s_k} \ar@{=}[d] & \cdot \ar@{->}[r]^{s_{k+1}} \ar@{->}[d]^{s_ks_{k+1}} & \cdot \ar@{->}[r]^{s_k} \ar@{->}[d]^{s_ks_{k+1}} & \cdot \ar@{=}[d] \\
  \cdot \ar@{->}[r]_{s_{k+1}} & \cdot \ar@{->}[r]_{s_{k}} & \cdot \ar@{->}[r]_{s_{k+1}} & \cdot
 }
 \]
 and
 \[
 \xymatrix{
  \cdot \ar@{->}[r]^{s_i} \ar@{=}[d] & \cdot \ar@{->}[d]_{s_is_j} \ar@{->}[r]^{s_j} & \cdot \ar@{=}[d] \\
  \cdot \ar@{->}[r]_{s_j}  & \cdot \ar@{->}[r]_{s_i} & \cdot
 }
\]
are commutative.

Consider the following set of reductions $S = \cup_{i=1}^{n-1}\{(s_is_i, 1)\}$ on the free monoid $W(S)$ generated by $s_1,\ldots, s_{n-1}$. Thus a word $U \in W(S)$ is reducible if and only if there is $U_1,U_2 \in W(S)$ such that $U = U_1s_is_iU_2$ for some $s_i \in \{s_1,\ldots, s_{n-1}\}.$ It is easy to see that $\thickapprox$ is $S$-admissible congruence.

Let us consider a partial preoder $<$ on $W(S)$ defined as follows: $U<V$ if and only if $\ell(U) < \ell(V)$, where $\ell(U)$ is a length of a word $U\in W(S).$ We see that if $U \thickapprox U'$ then $\ell(U) = \ell(U')$, and if $\mathfrak{r}:U \to V$, for some reduction $\mathfrak{r}$, then $\ell(U) > \ell(V).$ Thus, it is clear that $<$ is a $\thickapprox$-admissible monomial partial preodering on $W(S).$

Next, by the straightforward computations, it is easy to see that all ambiguities of $S$ are resolvable up to $\thickapprox$ and hence $(S) = \cup_{i=1}^{n-1}\{s_i^2-1\}$ is a $\thickapprox$-closed Gr\"obner--Shirshov basis with respect to the $<$. Thus, by Theorem \ref{CD-lemma}, the corresponding set $(\mathrm{Irr}(S),\thickapprox)$ is a set of squarefree words in $s_i$, $1\le i \le n-1.$
\begin{flushright}
$\square$
\end{flushright}
\end{example}

\begin{remark}\label{def_CD=CD}
 Since ``the usual equality'', \textit{i.e.,} $W \thickapprox W'$ if and only if $W=W'$ for all elements of a free category, is, of course, a congruence then $=$ replacing $\thickapprox$ we then get the classical Composition--Diamond Lemma for categories (see \cite[9.3]{Be}, \cite{BCL12}, \cite{BokSurv}). In such cases we shall say just a Gr\"obner--Shirshov basis instead of $=$-closed Gr\"onber--Shirshov basis.
\end{remark}

\begin{example}
  Let us consider the following graph $\Gamma$ shown below
  
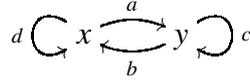
\[
   \xymatrix{
   x \ar@(ul,dl)_d \ar@/^/[r]^a  & y \ar@(ur,dr)^c \ar@/^/[l]^b 
   }
  \]
  
Set $\mathscr{C} = \mathsf{Cat}\langle \Gamma \, | \, d^2 =ab, c^2 = ba  \rangle$, and let $R$ be the ideal $(d^2 - ab, c^2-ba)$ in $\mathbb{K}\mathsf{Cat}\langle \Gamma \rangle$. Let $a<b<c<d$ and consider the corresponding deg-lex ordering on the free category $\mathsf{Cat}\langle \Gamma \rangle.$  

Put $\sigma = (d^2, ab)$, $\rho  = (c^2, ba)$. We have the following ambiguities $(\sigma, \sigma, d,d,d)$, $(\rho,\rho, c,c,c)$. We have 
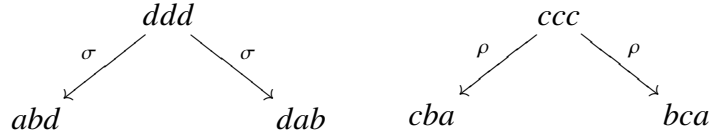
\[
 \xymatrix{
& ddd \ar@{->}[rd]^\sigma \ar@{->}[ld]_\sigma& \\
abd && dab
 }
 \qquad
 \xymatrix{
 & ccc \ar@{->}[rd]^\rho \ar@{->}[ld]_\rho & \\
 cba && bca
 }
\]

Thus, by the Buchberger--Shirshov algorithm, to calculate a Gr\"obner--Shirshov basis relative to the $\le$ we have to add to $S$ these two reductions $\mu_1=(dab, abd)$, $\mu_2 = (cba, bca)$. We have the following ambiguities $(\sigma, \rho_1, d,d,ab)$, $(\rho, \mu_2, c,c,ba)$.

We get
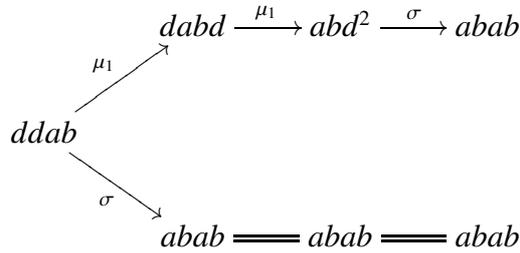
\[
 \xymatrix{
 & dabd \ar@{->}[r]^{\mu_1} & abd^2 \ar@{->}[r]^{\sigma} & abab \\
 ddab  \ar@{->}[ru]^{\mu_1} \ar@{->}[rd]_{\sigma}&&& \\
 & abab \ar@{=}[r] & abab \ar@{=}[r] & abab
 }
\]
and
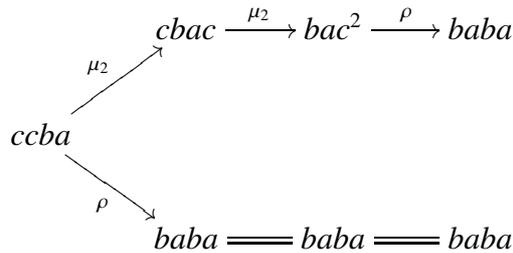
\[
 \xymatrix{
 & cbac \ar@{->}[r]^{\mu_2} & bac^2 \ar@{->}[r]^{\rho} & baba \\
 ccba  \ar@{->}[ru]^{\mu_2} \ar@{->}[rd]_{\rho}&&& \\
 & baba \ar@{=}[r] & baba \ar@{=}[r] & baba
 }
\]

Thus all ambiguities of the reduction system $S' = \{\sigma,\rho,\mu_1,\mu_2\}$ are resolvable, hence by the Composition--Diamond lemma the set $(d^2-ab, c^2-ba, dab-abd, cba-bca)$ is a Gr\"obner--Shirshov basis of the ideal $(R)$, and, therefore all basic elements of the category $\mathscr{C}$ (a $\mathbb{K}$-basis of the algebra $\mathbb{K}\mathscr{C}$) are elements of form $w \in \mathsf{Cat}\langle \Gamma \rangle$ such that $w \ne u w' v$ where $u,v,w'\in \mathsf{Cat}\langle \Gamma \rangle$ and $w' \ne d^2,c^2,dab,cba.$
\begin{flushright}
 $\square$
\end{flushright}
\end{example}

\section{Garside Theory and Gr\"obner--Shirshov basis}

This is a key section of this paper. We show how the main concepts of Garside theory can be obtained by using Theorem \ref{The_Main_Result}. We shall also see that in same cases the corresponding greedy normal form is exactly a Gr\"obner--Shirshov normal form. For a left-cancellative category and for its arbitrary subfamily we construct a set of reductions and we then show that this subfamily is Garside if and only if the set of reductions is confluent up to a congruence $\thickapprox$ (=deformations of paths by invertible elements).

We start with the following main construction which will be frequently used.

\begin{construction}\label{the_main_construction}
Let $\mathscr{C}$, $\mathscr{A}$ be categories with surjective functor $P:\mathscr{C} \to \mathscr{A}$, \textit{i.e.,} both its restrictions to objects and elements (=morphisms) are surjective. Let us assume that $\mathscr{A}$ has a presentation $\mathscr{A} = \mathsf{Cat} \langle \Gamma \, | \, R \rangle$. Consider a categorical order $\le$ on the free category $\mathsf{Cat}\langle \Gamma \rangle$. Let us assume that we know a Gr\"obner--Shirshov basis of the ideal $(R)$ relative to $\le$. Hence, by the Composition--Diamond lemma (= the classical verse, see Theorem \ref{CD-lemma} and Remark \ref{def_CD=CD}), we know a  basis $\mathrm{Irr}(R)$ (= a set of all irreducible elements) of $\mathscr{A}.$

Let us consider a map (=a section of $P$) $E: \mathscr{A} \to \mathscr{C}$ that is not a functor in general but $P\circ E = \mathrm{id}_\mathscr{C}$. For a given subset $\mathfrak{B} \subseteq \mathrm{Irr}(R)$ we construct the following germ
\[
 \Upsilon_E(\mathscr{C},\mathscr{A},P,\mathfrak{B}): = \bigcup_{a\in \mathfrak{B}} \left\{E_a, \, | \, E_a \bullet E_b:=E_{\mathrm{NF}(ab)} \in \Upsilon \mbox{ whenever $E(a \bullet b) = E(a) E(b)$}  \right\}.
\]

In the case $\mathscr{A} = \mathscr{C}$ and $P = \mathbf{id}_\mathscr{C}$ is an identity functor we then denote the germ $\Upsilon_E(\mathscr{C},\mathscr{C},\mathbf{id}_\mathscr{C},\mathfrak{B})$ by $\Upsilon_E(\mathscr{C},\mathfrak{B})$. We consider this partial case in details later in this section.

\end{construction}

\begin{definition}\label{order}
 For any $E_x, E_y \in \Upsilon$ we set $E_x \gtrapprox E_y$ if $E_x \preccurlyeq E_y$ and $E_y \not \preccurlyeq E_x$.
\end{definition}

\begin{lemma}
 If a germ $\Upsilon_E(\mathscr{C},\mathscr{A},P,\mathfrak{B})$ is left-associative then the binary operation $\gtrapprox$ is a partial preorder on the set $\{E_x, \, x \in \mathfrak{B}\}$.
\end{lemma}

\begin{proof}~\\
(1) It is clear that $E_a \gtrapprox E_a$ because of $a = a\mathbf{1}_{\mathfrak{s}(a)}$.~\\
(2) Let $E_a \gtrapprox E_b$ and $E_b \gtrapprox E_a$. Then we have $b = aa'$ and $a = bb'$ in $\mathscr{A}$ with $E_a\bullet E_{a'}, E_b \bullet E_{b'} \in \Upsilon$ that contradicts to the definition of $\gtrapprox$.~\\
(3) Let $E_a \gtrapprox E_b$ then $b = aa'$, $E_a\bullet E_{a'} \in \Upsilon$, there is no any relation $a = bb'$ in $\mathscr{A}$ with $E_b\bullet E_{b'}\in \Upsilon$. Let $E_b \gtrapprox E_c$ then $c =bb'$ with $E_b \bullet E_{b'} \in \Upsilon$ and there is no any relation $b=cc'$ in $\mathscr{A}$ with $E_c \bullet E_{c'} \in \Upsilon$.

It follows that $c = bb' = aa'b'$. $E_c = E_b\bullet E_{b'} = (E_a\bullet E_{a'})\bullet E_{b'}$ and by left-associativity of $\Upsilon$, $E_{a'} \bullet E_{b'} \in \Upsilon$. Hence $E_c = E_a \bullet E_{a'b'}.$

Assume now that $a = cc''$ with $E_c\bullet E_{c''} \in \Upsilon$. Then $b = aa'  = cc''a'$. $E_b = E_a\bullet E_{a'} = (E_c\bullet E_{c''})\bullet E_{a'}$ and by left-associativity of $\Upsilon$, $E_{c''}\bullet E_{a'} \in \Upsilon$ thus $E_{b} = E_{c}\bullet E_{\mathrm{NF}(c''a')}$ that gives a contradiction to the assumption $E_b \gtrapprox E_c$. It implies that $E_a \gtrapprox E_c.$ This completes the proof.
\end{proof}

\begin{definition}\label{reducible_elements} Let $\Upsilon = \Upsilon_E(\mathscr{C},\mathscr{A},P,\mathfrak{B})$ be a left-associative and left-cancellative germ. An element $E_xE_y\in \Upsilon^{[2]}$ is called \textit{reducible} if $y = y_1y_2$ with $E_x\bullet E_{y_1}, E_{y_1}\bullet E_{y_2} \in \Upsilon$ and $E_x \gtrapprox E_{\mathrm{NF}(xy_1)}$. We write $E_xE_y \to E_{\mathrm{NF}(xy_1)}E_{y_2}$ (of course, the case $E_{y_2} = \mathbf{1}_{\mathfrak{t}(y_1)}$ is also allowed). In otherwise the element $E_xE_y$ is called \textit{irreducible}.  
\end{definition}

\begin{construction}\label{reduction_system}
Keep the notations used in the previous Definition and Construction \ref{the_main_construction}. Introduce the following set of reductions 
\[
 S_\Upsilon(\mathfrak{B}):=\bigcup_{E_xE_y\in \Upsilon^{[2]}} \{\mathfrak{r}_{x,y}:E_xE_y \to E_{\mathrm{NF}(xy_1)}E_{y_2}\},
\]
and set
\[
 \overline{S}_\Upsilon(\mathfrak{B}):=\bigcup_{\substack{ E_xE_y\in \Upsilon^{[2]} \\ E_{\mathrm{NF}(xy_1)}E_{y_2} \in \mathrm{Irr}(S_\Upsilon(\mathfrak{B})) }}  \{ \mathfrak{r}_{x,y}:E_xE_y \to E_{\mathrm{NF}(xy_1)}E_{y_2}\}.
\]

It is clear that $\overline{S}_\Upsilon(\mathfrak{B}) \subseteq S_\Upsilon(\mathfrak{B}).$
\end{construction}

\begin{proposition}\label{J=invert}
Let us consider a germ $\Upsilon =\Upsilon_E(\mathscr{C},\mathscr{A},P,\mathfrak{B})$. An element $E_aE_b \in \Upsilon^{[2]}$ is irreducible if and only if the set $\mathcal{J}_\Upsilon(E_a,E_b)$ consists of invertible elements of $\Upsilon.$
\end{proposition}

\begin{proof}
 By Definition \ref{def_of_I&J},
 \[
  \mathcal{J}_\Upsilon(E_a,E_b): = \{E_c \in \Upsilon\, |\, E_a \bullet E_c \in \Upsilon, \, E_b = E_c \bullet E_d \mbox{ for some $E_d\in \Upsilon$} \}.
 \]
 
 We assume that $E_c \in \mathcal{J}_\Upsilon(E_a,E_b)$, $E_b = E_c\bullet E_d$ with $E_a \bullet E_c, E_c \bullet E_d \in \Upsilon$.
 
 (1) Let $E_aE_b$ be irreducible. By, $E_b = E_c \bullet E_d$, $E_a\bullet E_c \in \Upsilon$, and Definition \ref{reducible_elements}, $E_a$, $E_{\mathrm{NF}(ac)}$ must be not $\gtrapprox$-comparable. It follows that $E_{\mathrm{NF}(ac)}\preccurlyeq E_a$ because of $E_a \preccurlyeq E_{\mathrm{NF}(ac)}$. Hence $a = acc'$ with $E_{ac} \cdot E_{c'} \in \Upsilon.$ We have $E_a = E_{\mathrm{NF}(ac)} E_{c'} = (E_a E_c )E_{c'}$ and by left-associativity of $\Upsilon$, $E_c\bullet E_{c'} \in \Upsilon$. Hence, $E_a = E_{\mathrm{NF}(acc')} = E_a  (E_{c}E_{c'})$ and by left-concelativity of $\Upsilon$, $E_cE_{c'} = \mathbf{1}_{\mathfrak{s}(c)}.$
 
 Next, $E_aE_c = (E_{\mathrm{NF}(ac)}\bullet E_{c'})\bullet E_c$ an by left-associativity of $\Upsilon$, $E_{c'}E_c \in \Upsilon.$ Let us consider the element $E_cE_{c'}E_c$. We have $E_cE_{c'}E_c = (E_cE_{c'})E_c = E_c$ on the other have $E_cE_{c'}E_c = E_c (E_{c'}E_c)$ and by the left-cancelativity of $\Upsilon$, $E_{c'}E_c =  \mathbf{1}_{\mathfrak{s}(c')}$. Thus $\mathcal{J}_\Upsilon(E_a, E_b) \subseteq \Upsilon^\times$ as claimed.
 
 (2) Let $\mathcal{J}_\Upsilon(E_a,E_b) \subseteq \Upsilon^\times,$ for some element $E_aE_b \in \Upsilon^{[2]}$. Suppose that $E_aE_b \to E_{\mathrm{NF}(ac)}E_{d}$ where $E_{\mathrm{NF}(ac)}E_{d}$ is irreducible. It is clear $E_c \in \mathcal{J}_\Upsilon(E_a,E_b)$ and thus by assumption $E_c$ is invertible. Therefore $E_{\mathrm{NF}(ac)} \preccurlyeq E_a$ i.e,. $E_a \not\gtrapprox E_{\mathrm{NF}(ac)}$. This completes the proof. 
\end{proof}

\begin{lemma}
 Let $\thickapprox$ be a deformation of elements of $\mathsf{Cat}\langle\Upsilon\rangle$ by invertible elements. Then $\thickapprox$ is $\overline{S}_\Upsilon(\mathfrak{B})$-admissible.
\end{lemma}
\begin{proof}
We have to prove that if whenever $E_aE_b$ is reducible (\textit{resp.} irreducible) then any $E_{a'}E_{b'} \thickapprox E_aE_b$ is so.

By $E_aE_b \thickapprox E_{a'}E_{b'}$, $E_a = E_{a'}E_e$, $E_{b'} = E_bE_e$, where $E_e \in \Upsilon^\times$,

\[
 \xymatrix{
  \cdot \ar@{->}[r]^{E_{a'}} \ar@{=}[d] & \cdot \ar@{->}[r]^{E_{b'}} \ar@/_/[d]_{E_e} & \cdot \\
 \cdot \ar@{->}[r]_{E_a} & \cdot \ar@{->}[r]_{E_b} \ar@/_/[u]_{E_{e}^{-1}} & \cdot \ar@{=}[u]
 }
\]

(1) Let $E_aE_b$ be reducible, say $\mathfrak{r}_S: E_aE_b \to E_{\mathrm{NF}(ab_1)}E_{b_2}$. $\Upsilon \ni E_a\bullet E_{b_1} = (E_{a'}E_e)E_{b_1}$, by left-associativity of $\Upsilon$, $E_eE_{b_1} \in \Upsilon$. We thus have $E_{a'}\bullet E_e, E_{e}\bullet E_{b_1}, (E_{a'}\bullet E_e)\bullet E_{b_1} \in \Upsilon$. Hence, by Definition \ref{def_of_germ} (3), $E_{a'}\bullet( E_e\bullet E_{b_1}) = (E_{a'}\bullet E_e)\bullet E_{b_1} \in \Upsilon$, i.e,. $E_{a'}\bullet E_{\mathrm{NF}(eb_1)} \in \Upsilon$. Therefore, by $E_{b'}= E_{e}E_b = E_eE_{b_1}E_{b_2} = E_{\mathrm{NF}(eb_1)}E_{b_2}$, there is a reduction $\mathfrak{r}: E_{a'}E_{b'} \to E_{\mathrm{NF}(a'eb_1)}E_{b_2}$, i.e., $E_{a'}E_{b'}$ is reducible.

(2) Let $E_aE_b$ be irreducible. Then the statement immediately follows from the Proposition \ref{J=invert} and Definition \ref{def_of_deformation}, because of $E_aE_b = E_{a'}E_{b'}$ with $E_a= E_{a'}E_e$ and $E_b = E_{e}^{-1}E_{b'}$. 
\end{proof}

\begin{remark}
 Let us consider a germ $\Upsilon = \Upsilon_E(\mathscr{C},\mathscr{A},P,\mathfrak{B})$. It seems that we also have to add a set of reduction of the form $E_xE_x^{-1}=\mathbf{1}_{\mathrm{s}(x)}$ but it equivalences to $E_xE_x^{-1} \thickapprox \mathbf{1}_{\mathfrak{s}(x)}\cdot \mathbf{1}_{\mathfrak{s}(x)}$, indeed, 
\[
 \xymatrix{
  \cdot \ar@{->}[r]^{E_{x}} \ar@{=}[d] & \cdot \ar@{->}[r]^{E_{x}^{-1}} \ar@/_/[d]_{E_x^{-1}} & \cdot \\
 \cdot \ar@{->}[r]_{\mathbf{1}_{\mathfrak{s}(x)}} & \cdot \ar@{->}[r]_{\mathbf{1}_{\mathfrak{s}(x)}} \ar@/_/[u]_{E_{x}} & \cdot \ar@{=}[u]
 }
\]
thus according to Remark \ref{rem_of_empty_GSB} we do not do that. However, it is useful to consider an ambiguity of form $(E_a,E_b,E_c)$ with $E_b = E_{c}^{-1}$ and $\{E_aE_b \to E_{\mathrm{NF}(ab_1)}E_{b_2}\} \in S_\Upsilon(\mathfrak{B})$. First of all let us note that ${E_{b_1}} = E_{\mathrm{NF}(b_2c)}$. Indeed, we have $\Upsilon \ni \mathbf{1}_{\mathfrak{s}(b)} = E_b \bullet E_c = (E_{b_1}\bullet E_{b_2})\bullet E_c$, hence, by left left-associativity of $\Upsilon$, $E_{b_2}\bullet E_c \in \Upsilon$. Next, we have $E_b = E_{b_1} \bullet E_{b_2}, E_{b_2}\bullet E_c, E_b\bullet E_c = (E_{b_1}\bullet E_{b_2})\bullet E_c = \mathbf{1}_{\mathfrak{s}(b)} \in \Upsilon$ then by Definition \ref{def_of_germ} (3), $\mathbf{1}_{\mathfrak{s}(b)} = (E_{b_1}\bullet E_{b_2})\bullet E_c = E_{b_1}\bullet (E_{b_2}\bullet E_c)$, \textit{i.e.,} $E_{b_1} = E_{\mathrm{NF}(b_2c)}^{-1}$, as claimed. Thus we have 
\[
 \xymatrix{
  \cdot \ar@{->}[r]^{E_{a}} \ar@{=}[d] & \cdot \ar@{->}[rr]^{\mathbf{1}_{\mathfrak{s}(b_1)}} \ar@/_/[d]_{E_{b_1}} && \cdot \\
 \cdot \ar@{->}[r]_{E_{\mathrm{NF}(ab_1)}} & \cdot \ar@{->}[rr]_{E_{\mathrm{NF}(b_2c)}} \ar@/_/[u]_{E_{\mathrm{NF}(b_2c)}} && \cdot \ar@{=}[u]
 }
\]
\textit{i.e.,} $E_a\cdot \mathbf{1}_{\mathfrak{s}(b)} \thickapprox E_{\mathrm{NF}(ab_1)} \cdot E_{\mathrm{NF}(b_2c)}$, because of $\mathfrak{s}(b) = \mathfrak{s}(b_1)$. On the other hand, the ambiguity $(E_a,E_b,E_c)$ gives
\[
 \xymatrix{
 & E_aE_bE_c \ar@{->}[rd] \ar@{->}[ld] & \\
E_{\mathrm{NF}(ab_1)}E_{b_2}E_c \ar@{->}[d] && E_a \mathbf{1}_{\mathfrak{s}(b)} \ar@{=}[d] \\
E_{\mathrm{NF}(ab_1)} E_{\mathrm{NF}(b_2c)} && E_a \mathbf{1}_{\mathfrak{s}(b)}
 }
\]
\textit{i.e.,} according to the $\thickapprox$-version of the Buchberger--Shirshov algorithm (see Remark \ref{Buch-Shirsh}) we get a new reduction $E_{\mathrm{NF}(ab_1)} E_{\mathrm{NF}(b_2c)} \to E_a \mathbf{1}_{\mathfrak{s}(b)}$ which is equivalent to $E_{\mathrm{NF}(ab_1)} E_{\mathrm{NF}(b_2c)} \thickapprox E_a \mathbf{1}_{\mathfrak{s}(b)}.$
 \begin{flushright}
 $\square$
 \end{flushright}
\end{remark}

\begin{lemma}\label{Y=Garside<=>total_pre_order}
Let $\Upsilon = \Upsilon_E(\mathscr{C},\mathscr{A},P,\mathfrak{B})$ be a left-associative and left-cancellative germ, then $\Upsilon$ is a Garside germ if and only if a set $\mathcal{I}(E_x, E_y)$ is totally preordered by $\preccurlyeq$.
\end{lemma}
\begin{proof}
(1) Let $\mathcal{I}(E_x, E_y)$ be a totally preordered set by $\preccurlyeq$. Define a map $\mathfrak{I}: \Upsilon^{[2]} \to \Upsilon$ as follows
\begin{equation}\label{I(Ea,Eb)}
    \mathfrak{I}(E_a,E_b) : = \begin{cases}
    E_{\mathrm{NF}(ab_1)}, & \mbox{if $E_aE_b \to E_{\mathrm{NF}(ab_1)}E_{b_2}$, where $E_{\mathrm{NF}(ab_1)}E_{b_2}$ is irreducible,} \\
    E_a, & \mbox{in otherwise.}
   \end{cases}    
\end{equation}

Let $E_c \in \mathcal{I}(E_a,E_b)$, then $E_c = E_{\mathrm{NF}(ab')}$ where $E_b = E_{b'}E_{b''}$. Then either $E_c \preccurlyeq E_{\mathrm{NF}(ab_1)}$ or $E_{\mathrm{NF}(ab_1)} \preccurlyeq E_c$.

Let $E_{\mathrm{NF}(ab_1)} \preccurlyeq E_c$, say $E_c = E_{\mathrm{NF}(ab_1)} E_{d}$. Since $E_c = E_{\mathrm{NF}(ab')}$ then $E_{a}E_{b'} = E_{a}E_{b_1}E_d$, by left-cancelativity of $\Upsilon$, $E_{b'} = E_{b_1}E_d$. Next, we have $E_b = E_{b_1}E_{b_2} = E_{b'}E_{b''}$ then $E_{b_2} = E_dE_{b''}$. Since $E_{\mathrm{NF}(ab_1)}E_{b_2}$ is assumed to be irreducible then by Proposition \ref{J=invert}, $E_d$ is an invertible element, because of $E_{\mathrm{NF}(ab_1)}E_{b_2} = E_{\mathrm{NF}(ab_1)}E_dE_{b''}$ and $\Upsilon \ni E_{\mathrm{NF}(ab')}= E_{\mathrm{NF}(ab_1)}E_d$. It follows that $E_c =^\times E_{\mathrm{NF}(ab_1)}$ and hence $E_c \preccurlyeq E_{\mathrm{NF}(ab_1)}$ for any $E_c \in \mathcal{I}(E_a,E_b)$. Thus $\mathfrak{I}$ is a greatest $\mathcal{I}$-function on $\Upsilon.$ Hence $\Upsilon$ is a Garside germ.

(2) Conversely, let $\Upsilon$ be a Garside germ. We then get a function $\mathfrak{I}$ obeying $\mathcal{I}$-law, say, $\mathfrak{I}$, \textit{i.e.,} we have $\mathfrak{I}(E_a \bullet E_b, E_c) = \mathfrak{I}(E_a,\mathfrak{I}(E_b,E_c))$, where $E_a\bullet E_b \in \Upsilon.$ We have
\[
 \xymatrix{
 & E_aE_bE_c \ar@{->}[rd] \ar@{->}[ld]& \\
 E_{\mathrm{NF}(ab)}E_c \ar@{->}[d] && E_a E_{\mathrm{NF}(bc_3)}E_{c_4} \ar@{->}[d] \\
 E_{\mathrm{NF}(abc_1)}E_{c_2} && E_{\mathrm{NF}(ad_1)}E_{d_2}E_{c_4}
 }
\]
then $E_{\mathrm{NF}(abc_1)} = E_{\mathrm{NF}(ad_1)}$, where $E_{\mathrm{NF}(bc_3)} = E_{d_1}E_{d_2}.$ By the left-associativity of $\Upsilon$, $E_{\mathrm{NF}(bc_1)} = E_{d_1}.$ Next, we have $E_{b}E_{c_3} = E_{d_1}E_{d_2} = E_{b}E_{c_1}E_{d_2}$ then $E_{c_3} = E_{\mathrm{NF}(c_1d_2)}.$ Further, by $E_c = E_{c_1}E_{c_2} = E_{c_3}E_{c_4}$, $E_{c_2} = E_{d_2}E_{c_4}$. We thus get $c_3  = c_1d_2$ and $c_2 = d_2c_4$ and the statement follows.
\end{proof}

\begin{lemma}\label{total_pre-order=>condition}
If for any $E_a,E_b \in \Upsilon$ the set $(\mathcal{I}(E_a,E_b),\preccurlyeq)$ is totally preordered by $\preccurlyeq$ then all elements of $\Upsilon^{[2]}$ are reduction unique up to $\thickapprox$ under $S_\Upsilon(\mathfrak{B})$.

\end{lemma}
\begin{proof}
 In other words we have to prove that the following diagrams
 \[
   \xymatrix{
    E_xE_{\mathrm{NF}(yz)} \ar@{->}[r]^{\mathfrak{r}_{x,yz}} \ar@{->}[d]_{\overline{\mathfrak{r}_{x, yz}}} \ar@{->}[dr]^{\overline{\mathfrak{r}_{xy_1,y_2z}}\circ \mathfrak{r}_{x,yz}} & E_{\mathrm{NF}(xy_1)}E_{\mathrm{NF}(y_2z)} \ar@{->}[d]^{\overline{\mathfrak{r}_{xy_1,y_2z}}}  \\
    E_{\mathrm{NF}(xw_1)}E_{w_2} \ar@{->}[r]_{\thickapprox} & E_{\mathrm{NF}(xy_1d_1)}E_{d_2}
   }
   \qquad  \xymatrix{
   E_{y_2}E_{z} \ar@{->}[r]^{{\mathfrak{r}_{y_2,z}}} \ar@{->}[rd]^{\overline{\mathfrak{r}_{y_2z_1,z_2}}\circ \mathfrak{r}_{y_2,z}} \ar@{->}[d]_{\overline{\mathfrak{r}_{y_2,z}}} & E_{\mathrm{NF}(y_2z_1)}E_{z_2} \ar@{->}[d]^{\overline{\mathfrak{r}_{y_2z_1, z_2}}} \\
   E_{\mathrm{NF}(y_2z')}E_{z''} \ar@{->}[r]_{\thickapprox}& E_{\mathrm{NF}(y_2z_1z_{21})}E_{z_{22}}
   }
  \]
 are commutative, where all $\overline{\mathfrak{r}}$ are elements of $\overline{S}_\Upsilon(\mathfrak{B})$.
 
 (1) By the assumption the set $(\mathcal{I}(E_x, E_{\mathrm{NF}(yz)}), \preccurlyeq)$ is totally preordered. We have $E_{w_1}, E_{\mathrm{NF}(y_1d)} \in \mathcal{I}(E_x, E_{\mathrm{NF}(yz)})$.
 
\textit{Case 1.} Let $E_{w_1} \preccurlyeq E_{\mathrm{NF}(y_1d)}$ then $E_{\mathrm{NF}(y_1d)} = E_{w_1}E_e$. Since $E_{\mathrm{NF}(yz)}$, $E_{\mathrm{NF}(y_2z)}$ have been assumed to be equal to $E_{w_1}E_{w_2}$ and $E_{d_1}E_{d_2}$ respectively then we get 
 \begin{eqnarray*}
  E_{w_1}E_{w_2} &=& E_{\mathrm{NF}(y_1d)} = E_{y_1}E_{y_2}E_{z} = E_{y_1}E_{\mathrm{NF}(y_2)z} \\
  &=& E_{y_1}E_{d_1}E_{d_2} = E_{\mathrm{NF}(y_1d_1)}E_{d_2} = E_{w_1}E_eE_{d_2},
 \end{eqnarray*}
hence, by left-cancellativity of $\Upsilon$, $E_{w_2} = E_eE_{d_2}$, \textit{i.e.,} $E_e \preccurlyeq E_{w_2}$. Further, $\Upsilon \ni E_{\mathrm{NF}(xy_1d_1)} = E_{x}E_{\mathrm{NF}(y_1d_1)} = E_{x}E_{w_1}E_e$. Since $E_{w_1}E_e = E_{\mathrm{NF}(y_1d_1)}$ then $E_{w_1}\bullet E_e \in \Upsilon$. It follows that we get a reduction $E_{\mathrm{NF}(xw_1)}E_{w_2} \to E_{\mathrm{NF}(xw_1e)}E_d$. But $E_{\mathrm{NF}(xw_1)}E_{w_2}$ is assumed to be irreducible then by Proposition \ref{J=invert}, $E_e$ is invertible. We thus obtain
\[
 E_{\mathrm{NF}(xy_1d_1)} = E_xE_{\mathrm{NF}(y_1d_1)} = E_{x}E_{w_1}E_{f} = E_{\mathrm{NF}(xw_1)}E_e,
\]
therefore $E_{\mathrm{NF}(xy_1d_1)} =^\times E_{\mathrm{NF}(xw_1)}$, we thus get
\[
 \xymatrix{
  \cdot \ar@{->}[rr]^{E_{\mathrm{NF}(xw_1)}} \ar@{=}[d] && \cdot \ar@{->}[r]^{E_{w_2}} \ar@/_/[d]_{E_e} & \cdot \\
 \cdot \ar@{->}[rr]_{E_{\mathrm{NF}(xy_1d_1)}} && \cdot \ar@{->}[r]_{E_{d_2}} \ar@/_/[u]_{E_{e}^{-1}} & \cdot \ar@{=}[u]
 }
\]
as claimed.

\textit{Case 2.} Let $E_{\mathrm{NF}(y_1d_1)} \preccurlyeq E_{w_1}$, say $E_{w_1} = E_{\mathrm{NF}(y_1d_1)}E_e.$ By the left-associativity of $\Upsilon$, $E_{d_1}\bullet E_e \in \Upsilon$. Since we have assumed that $E_{\mathrm{NF}(yz)} = E_{w_1}E_{w_2}$, $E_{\mathrm{NF}(y_2z)} = E_{d_1}E_{d_2}$ then
\[
 E_{\mathrm{NF}(yz)} = E_{w_1}E_{w_2} = E_{\mathrm{NF}(y_1d_1)}E_eE_{w_2} = E_{y_1}E_{d_1}E_eE_{w_2}
\]
by $E_y = E_{y_1}E_{y_2}$ and left-associativity of $\Upsilon$, $E_{\mathrm{NF}(y_2z)} = E_{\mathrm{NF}(d_1e)}E_{w_2}$. $E_{\mathrm{NF}(y_2z)} = E_{d_1}E_{d_2}$ implies that $E_{d_2} = E_eE_{w_2}.$ Next, by $E_{\mathrm{NF}(xw_1)} = E_{\mathrm{NF}(xy_1d_1)}E_e$ and $E_{e} \preccurlyeq E_{d_2}$ then there is a reduction $E_{\mathrm{NF}(xy_1d_1)}E_{d_2} \to E_{\mathrm{NF}(xy_1d_1e)}E_{w_2}$. Hence, by Proposition \ref{J=invert}, $E_e$ is invertible because of $E_{\mathrm{NF}(xy_1d_1)}E_{d_2}$ has been assumed to be irreducible. Thus $E_{\mathrm{NF}(xw_1)} =^{\times} E_{\mathrm{NF}(xy_1d_1)}$, we obtain
\[
 \xymatrix{
  \cdot \ar@{->}[rr]^{E_{\mathrm{NF}(xw_1)}} \ar@{=}[d] && \cdot \ar@{->}[r]^{E_{w_2}} \ar@/^/[d]^{E_e^{-1}} & \cdot \\
 \cdot \ar@{->}[rr]_{E_{\mathrm{NF}(xy_1d_1)}} && \cdot \ar@{->}[r]_{E_{d_2}} \ar@/^/[u]^{E_{e}} & \cdot \ar@{=}[u]
 }
\]
and then the statement follows.

(2) We have $E_{z'}, E_{\mathrm{NF}(z_1z_{21})} \in \mathcal{I}(E_{y_2}, E_z)$.

\textit{Case 1.} Let $E_{\mathrm{NF}(z_1z_{21})} \preccurlyeq E_{z'}$, \textit{i.e.,} $E_{z'} = E_{\mathrm{NF}(z_1z_{21})} E_e$, hence by left-cancellativity of $\Upsilon$ and $E_z = E_{z'}E_{z''} = E_{\mathrm{NF}(z_1z_{21})}E_{z_{22}}$, $E_{z_{22}} = E_eE_{z''}$. Since $E_{\mathrm{NF}(y_2z')} = E_{\mathrm{NF}(y_2z_1z_{21}e)}$ then, by Proposition \ref{J=invert}, $E_e \in \Upsilon^\times.$ It implies that 
\[
 \xymatrix{
  \cdot \ar@{->}[rr]^{E_{\mathrm{NF}(y_2z_1z_{21})}} \ar@{=}[d] && \cdot \ar@{->}[r]^{E_{z_{22}}} \ar@/^/[d]^{E_e} & \cdot \\
 \cdot \ar@{->}[rr]_{E_{y_2z'}} && \cdot \ar@{->}[r]_{E_{z''}} \ar@/^/[u]^{E_{e}^{-1}} & \cdot \ar@{=}[u]
 }
\]
and the statement follows.
\end{proof}

\textit{Case 2.} $E_{z'} \preccurlyeq E_{\mathrm{NF}(z_1z_{21})}$, \textit{i.e.,} $E_{\mathrm{NF}(z_1z_{21})} = E_{z'}E_e$, then by the left-associativity of $\Upsilon$ and $E_z = E_{z'}E_{z''} = E_{\mathrm{NF}(z_1z_{21})}E_{z_22}$, $E_{z''} = E_eE_{z_{22}}$. Since $E_{\mathrm{NF}(y_2z_1z_{21})} = E_{\mathrm{NF}(y_2z'e)}$ then by Proposition \ref{J=invert}, $E_e \in \Upsilon^\times$, it implies that 
\[
 \xymatrix{
  \cdot \ar@{->}[rr]^{E_{\mathrm{NF}(y_2z_1z_{21})}} \ar@{=}[d] && \cdot \ar@{->}[r]^{E_{z_{22}}} \ar@/^/[d]^{E_e^{-1}} & \cdot \\
 \cdot \ar@{->}[rr]_{E_{y_2z'}} && \cdot \ar@{->}[r]_{E_{z''}} \ar@/^/[u]^{E_{e}} & \cdot \ar@{=}[u]
 }
\]
and thus the statement follows.

\begin{lemma}\label{(S)=GSB}
If a set $(\mathcal{I}(E_a,E_b),\preccurlyeq)$ is totally preordered for any $E_a,E_b \in \Upsilon$ then all ambiguities of $\overline{S}_\Upsilon(\mathfrak{B})$ are $\thickapprox$-resolvable. 
\end{lemma}

\begin{proof}
Let us consider an ambiguity $(\sigma,\tau, E_{a}, E_b,E_c)$, we have
\[
 \xymatrix{
 & E_aE_bE_c \ar@{->}[rd]^{\mathfrak{r}_{b,c}} \ar@{->}[ld]_{\mathfrak{r}_{a,b}}& \\
E_{\mathrm{NF}(ab_1)}E_{b_2}E_c && E_a E_{\mathrm{NF}(bc_1)}E_{c_2}
 }
\]

Using Lemma \ref{total_pre-order=>condition}, we get the following commutative diagrams
\[
 \xymatrix{
  E_{\mathrm{NF}(ab_1)}E_{b_2}E_c \ar@{->}[d]_{\overline{\mathfrak{r}_{b_2,c}}} \ar@{->}[r]^{\mathfrak{r}_{b_2,c}} & E_{\mathrm{NF}(ab_1)}E_{\mathrm{NF}(b_2c_1)}E_{c_2} \ar@{->}[d]^{\overline{\mathfrak{r}_{b_2c_1,c_2}}}\\
  E_{\mathrm{NF}(ab_1)}E_{\mathrm{NF}(b_2c')}E_{c''} \ar@{->}[r]_{\thickapprox} & E_{\mathrm{NF}(ab_1)}E_{\mathrm{NF}(b_2c_1c_{21})}E_{c_{22}}
 }
\]
and
\[
 \xymatrix{
  E_aE_{\mathrm{NF}(bc_1)}E_{c_2} \ar@{->}[d]_{\overline{\mathfrak{r}_{a,bc_1}}} \ar@{->}[r]^{{\mathfrak{r}_{a,bc_1}}} & E_{\mathrm{NF}(ab_1)}E_{\mathrm{NF}(b_2c_1)}E_{c_2} \ar@{->}[d]^{\overline{\mathfrak{r}_{ab_1,b_2c_1}}} \\
  E_{\mathrm{NF}(ad_1)}E_{d_2}E_{c_2} \ar@{->}[r]_{\thickapprox} & E_{\mathrm{NF}(ab_1f_1)}E_{f_2}E_{c_2}
 }
\]

Let us consider $E_{\mathrm{NF}(ab_1)} E_{\mathrm{NF}(b_2c_1)}E_{c_2}$. Since $E_{\mathrm{NF}(b_2c_1c_{21})} = E_{f_1}E_{f_2}E_{c_{21}}$, $E_{\mathrm{NF}(ab_1)}\bullet E_{f_1} \in \Upsilon$ we then have reductions $\mathfrak{r}_1: E_{f_2}E_{c_2} \to E_{\mathrm{NF}(f_2c_{21})}E_{c_{22}}$, and $\mathfrak{r}_2:E_{\mathrm{NF}(ab_1)}E_{\mathrm{NF}(b_2c_1c_{21})} \to E_{\mathrm{NF}(ab_1f_1)}E_{\mathrm{NF}(f_2c_{21})}$. Next, by left-associativity of $\Upsilon$, $E_{f_2}\bullet E_{c_{21}} \in \Upsilon$ and we get a reduction $\mathfrak{r}_3:E_{f_2}E_{c_2} \to E_{\mathrm{NF}(f_2c_{21})}E_{c_{22}}$. Thus we get
\[
 \xymatrix{
 E_{\mathrm{NF}(ab_1)} E_{\mathrm{NF}(b_2c_1)}E_{c_2}  \ar@{->}[r]^{\overline{\mathfrak{r}_{b_2c_1,c_2}}} \ar@{->}[d]_{\mathfrak{r}_2} & E_{\mathrm{NF}(ab_1)}E_{\mathrm{NF}(b_2c_1c_{21})}E_{c_{22}} \ar@{->}[d]^{\mathfrak{r}_1} \\
 E_{\mathrm{NF}(ab_1f_1)}E_{f_2}E_{c_2} \ar@{->}[r]_{\mathfrak{r}_3} & E_{\mathrm{NF}(ab_1f_1)}E_{\mathrm{NF}(f_2c_{21})}E_{c_{22}}
 }
\]

Therefor we obtain the following commutative diagram up to $\thickapprox$

\begin{equation*}
    \xymatrix{
 &E_aE_bE_c \ar@{->}[rd] \ar@{->}[ld] & \\
 E_{\mathrm{NF}(ab_1)}E_{b_2}E_c \ar@{->}[d] && E_a E_{\mathrm{NF}(bc_1)}E_{c_2} \ar@{->}[d] \\
 E_{\mathrm{NF}(ab_1)}E_{\mathrm{NF}(b_2c')}E_{c''} \ar@{.>}[rd]^{\rightsquigarrow} && E_{\mathrm{NF}(ad_1)}E_{d_2}E_{c_2} \ar@{.>}_{\leftsquigarrow}[ld] \\
 & E_{\mathrm{NF}(abf_1)}E_{\mathrm{NF}(f_2c_{21})}E_{c_{22}} & 
}
\end{equation*}
where the left and right dotted arrows denote the reductions
\begin{eqnarray*}
    E_{\mathrm{NF}(ab_1)}E_{\mathrm{NF}(b_2c')}E_{c''} &\rightsquigarrow& E_{\mathrm{NF}(abf_1)}E_{\mathrm{NF}(f_2c_{21})}E_{c_{22}}, \\
    E_{\mathrm{NF}(ad_1)}E_{d_2}E_{c_2} &\rightsquigarrow& E_{\mathrm{NF}(abf_1)}E_{\mathrm{NF}(f_2c_{21})}E_{c_{22}},
\end{eqnarray*}
respectively. It implies that an ambiguity $(\sigma, \tau, E_a, E_b,E_c)$ is resolvable up to $\thickapprox,$ as claimed.
\end{proof}

\begin{theorem}[{\bf The Main Result}]\label{The_Main_Result}
Let $\mathscr{C},\mathscr{A}$ be categories, $P:\mathscr{C} \to \mathscr{A}$ a surjective functor, $E$ its section, and the corresponding germ $\Upsilon = \Upsilon_E(\mathscr{C},\mathscr{A},P,\mathfrak{B})$ (see Construction \ref{the_main_construction}) left-cancelative and left-associative.

Let $\thickapprox$ be a deformation of paths by invertible elements, and consider the correspond set $\overline{S}_\Upsilon(\mathfrak{B})$ of reductions (see Construction \ref{reduction_system}).

The germ $\Upsilon$ is a Garside germ if and only if all ambiguities of the set $\overline{S}_\Upsilon(\mathfrak{B})$ are resolvable up to $\thickapprox.$

Moreover, the set of all irreducible elements $\mathrm{Irr}(\overline{S}_\Upsilon(\mathfrak{B}))$ coincides with the set of all $\Upsilon$-normal elements of the germ $\Upsilon$, and any two $\Upsilon$-normal decompositions of an element of $\mathsf{Cat}\langle \Upsilon \rangle$ are $\thickapprox$-equivalent (i.e., are $\Upsilon^\times$-deformations of one another).

Finally, if $\Upsilon$ is a solid generating family in $\mathscr{C}$ then $\mathscr{C} \cong \mathsf{Cat}\langle \Upsilon \rangle$.

\end{theorem}
\begin{proof} Let us prove the first statement.

(1) Let $\Upsilon$ be a Garside germ. Then by Lemma \ref{Y=Garside<=>total_pre_order}, the set $\mathcal{I}(E_a,E_b)$ is totally preordered by $\preccurlyeq$ for any $E_a,E_b$, thus, by Lemma \ref{total_pre-order=>condition} and Lemma \ref{(S)=GSB} the statement follows.
 
(2) Conversely, let all ambiguities of the set $\overline{S}_\Upsilon(\mathfrak{B})$ are $\thickapprox$-resolvable. Define a map $\mathfrak{I}: \Upsilon^{[2]} \to \Upsilon$ as follows
\[
 \mathfrak{I}(E_a,E_b) : = \begin{cases}
    E_{\mathrm{NF}(ab_1)}, & \mbox{if $E_aE_b \to E_{\mathrm{NF}(ab_1)}E_{b_2}$ and $E_{\mathrm{NF}(ab_1)}E_{b_2}$ is irreducible,} \\
    E_a, & \mbox{in otherwise.}
   \end{cases}
\]
  
Let $E_c \in \mathcal{I}(E_a,E_b)$, then $E_c = E_{\mathrm{NF}(ab')}$ where $E_b = E_{b'}E_{b''}$. Consider $E_aE_b$, by assumption, this element is reduction-unique up to $\thickapprox$ under $\overline{S}_\Upsilon(\mathfrak{B})$, \textit{i.e.,} we have the following commutative diagram
\[
  \xymatrix{
    E_aE_{b} \ar@{->}[r]^{\mathfrak{r}_{a,b}} \ar@{->}[d]_{\overline{\mathfrak{r}_{a, b}}} & E_{\mathrm{NF}(ab')}E_{b''} \ar@{->}[d]^{\overline{\mathfrak{r}_{ab',b''}}}  \\
    E_{\mathrm{NF}(ab_1)}E_{b_2} \ar@{->}[r]_{\thickapprox} & E_{\mathrm{NF}(ab'd_1)}E_{d_2}
   }
\]
where $E_{\mathrm{NF}(ab'd_1)}E_{d_2}$ is assumed to be irreducible, $E_{b''} = E_{d_1}E_{d_2}$, $E_{\mathrm{NF}(ab')}\bullet E_{d_2} \in \Upsilon$. Thus we have $E_{\mathrm{NF}(ab_1)} =  E_{\mathrm{NF}(ab'd_1)} E_{e}$, for some $E_e \in \Upsilon^\times.$ We have $E_{\mathrm{NF}(ab'd_1)} E_{e} = E_{\mathrm{NF}(ab')}E_{d_1}E_e$, i.e,. $E_{\mathrm{NF}(ab_1)} = E_{\mathrm{NF}(ab')}E_{d_1}E_e$, hence $E_{\mathrm{NF}(ab')} \preccurlyeq E_{\mathrm{NF}(ab_1)}$. Thus $\mathfrak{I}$ is a greatest $\mathcal{I}$-function on $\Upsilon.$ Hence $\Upsilon$ is a Garside germ.

The last statement immediately follows from Proposition \ref{deform_of_S-greedy}, Proposition \ref{greedy} (3), Proposition \ref{J=invert}, and Theorem \ref{germ_from_Garside}.
\end{proof}

\begin{corollary}\label{GSB=greedy}
 If a preorder $\gtrapprox$ has a descending chain condition then $\Upsilon = \Upsilon_E(\mathscr{C},\mathscr{A},P,\mathfrak{B})$ is a Garside germ if and only if the corresponding set of polynomials $(\overline{S}_\Upsilon(\mathfrak{B})) = \{E_xE_y-E_{\mathrm{NF}(xy_1)}E_{y_2} \}$, where $E_x,E_y$ run over all generators of the $\Upsilon$, and whenever all $E_{\mathrm{NF}(xy_1)}E_{y_2}$ is irreducible, is a $\thickapprox$-closed Gr\"obner--Shirshov basis relative to $\gtrapprox$ (see Definition \ref{order}) and then the corresponding Gr\"obner--Shirshov normal form is exactly the $\Upsilon$-normal form.
\end{corollary}
\begin{proof}
 It immediately follows from Theorems \ref{CD-lemma}, \ref{The_Main_Result}.
\end{proof}

\begin{remark}
Consider a germ $\Upsilon = \Upsilon_E(\mathscr{C},\mathscr{A}, P,\mathfrak{B})$ is assumed to be a left-associative and left-cancellative. We aim to show that a deformations of paths by invertible elements can be also deduced as follows. Let us assume that we should use the classical version of Composition--Diamond lemma. Consider an ambiguity, say, $(\sigma, \tau, E_a,E_b,E_c)$, with $E_a\bullet E_b, E_b\bullet E_c \in \Upsilon.$ We then get 
\[
 \xymatrix{
 & E_aE_bE_c \ar@{->}[rd]^{\mathfrak{r}_{\sigma}} \ar@{->}[ld]_{\mathfrak{r}_{\tau}}& \\
E_{\mathrm{NF}(ab)}E_c && E_a E_{\mathrm{NF}(bc)}
 }
\]
hence, adding this relation $E_a E_{\mathrm{NF}(bc)} = E_{\mathrm{NF}(ab)}E_c$, and if $E_a \gtrapprox E_{\mathrm{NF}(ab)}$ we then get a reduction mentioned before. But let us assume now that $E_a \not\gtrapprox E_{\mathrm{NF}(ab)}$, thus we get $E_{a} =^\times E_{\mathrm{NF}(ab)}$, say $E_a = E_{\mathrm{NF}(ab)}E_e$. Hence, by left-cancelativity of $\Upsilon$, $E_bE_e =  \mathbf{1}_{\mathfrak{s}(b)}$, then $E_b = E_e^{-1},$ \textit{i.e.,} $E_e \in \Upsilon^\times.$ Then adding new relation we get deformations of paths. It is clear that we can, without loss of generality, assume that $E_{\mathrm{NF}(ab)}E_c$, $E_aE_{\mathrm{NF}(bc)}$ are assumed to be irreducible. 
\end{remark}

\subsection{The Case $\mathscr{C} = \mathscr{A}$, $P=\mathbf{id}_\mathscr{C}$}\label{C=A}

Let us consider a left-cancellative category $\mathscr{C}$ with a presentation $\mathscr{C}= \mathrm{Cat}\langle \Gamma\, |\, R \rangle$ and let $\thickapprox$ be a deformation of paths by invertible elements. Let $S_R$ be a set of reductions on $\mathsf{Cat}\langle \Gamma \rangle$ such that the corresponding set of polynomials $(S_R)$ is to be a $\thickapprox$-closed Gr\"obner--Shirshov basis of the ideal $I(R)$ of the category algebra $\mathbb{K}\mathscr{C}$ relative to a $\thickapprox$-admissible categorical preorder $\lessapprox$ compatible with $S_R$ having descending chain condition. Thus, by Theorem \ref{CD-lemma}, a basis $\mathfrak{B} = \mathfrak{B}(\mathscr{C})$ of $\mathscr{C}$ is the set $\mathrm{Irr}(S_R,\thickapprox)$ and $\mathscr{C} \cong \mathsf{Cat}\langle \Gamma \, | \, a\cdot b = \mathfrak{r}_{S_R}(ab) \rangle$.

Next, for an identity functor $\mathbf{id}_\mathscr{C}:\mathscr{C} \to \mathscr{C}$ we consider its arbitrary section $E$ and by Construction \ref{the_main_construction} we construct the corresponding germ $\Upsilon_E(\mathscr{C},\mathfrak{B}')$ and the corresponding set of reductions $\overline{S}_\Upsilon(\mathfrak{B}')$ for some $\mathfrak{B}' \subseteq \mathfrak{B}$. It is easy to see that for an identity functor $\mathbf{id}_\mathscr{C}:\mathscr{C} \to \mathscr{C}$ considering its section $E$ is the same as choosing the corresponding subset of basis elements of $\mathscr{C}$, hence it is enough to consider a subset $\mathfrak{B}'$ of basis elements of $\mathscr{C}$. We thus denote the corresponding germ by $\Upsilon(\mathscr{C},\mathfrak{B}')$, and instead of $E_a$ we just write $a$ for any $a \in \mathrm{Irr}(S_R).$

Thus, by Theorem \ref{The_Main_Result}, we can say that there is the following one-to-one correspondence
\[
 \begin{pmatrix*}[l]
   \mbox{Garside families in } \\ \mbox{a left-cancellative category $\mathscr{C}$} \\\mbox{with a presentation $\mathscr{C} = \mathrm{Cat}\langle \Gamma\, |\, R \rangle$.}
 \end{pmatrix*}
 \rightleftarrows \begin{pmatrix*}[l]
   \mbox{subsets $\mathfrak{B}' \subseteq \mathrm{Irr}(S_R)$ such that } \\ \mbox{the corresponding set $\overline{S}_\Upsilon(\mathfrak{B}')$ of reductions }\\
   \mbox{is to be confluent up to $\thickapprox$.}
 \end{pmatrix*}
\]

\begin{example}[\textbf{The Klein bottle monoid, (cf. \cite[A, I. 3.2 and Example IV, 2.35]{Dehbook})}]\label{the_Klein_bottle_monoid}
  Let us consider the Klein bottle monoid 
  \[
   K^+ = \mathsf{Smg}\langle a, b \, | \, bab = a\rangle
   \]
 
Set $a<b$ and consider the corresponding deg-lex order $\le$ on the free monoid $W=W(a,b)$ generated by $a,b$. We have only one ambiguity $(\rho,\rho,ba,b,ab)$, where $\rho = (bab, a)$. We have
\[
 \xymatrix{
  & babab \ar@{->}[ld]_\rho  \ar@{->}[rd]^\rho& \\
  a^2b && ba^2
 }
\]
hence we get a new relation $\sigma = (ba^2, a^2b)$ and then the ambiguity $(\rho,\sigma, ba,b,a^2)$. We obtain
\[
 \xymatrix{
& baa^2b \ar@{=}[r] & ba^2 ab \ar@{->}[r]^{\sigma}& a^2bab \ar@{->}[r]^\rho & a^2a \ar@{=}[r] & a^3 \\
baba^2 \ar@{->}[ru]^{\sigma} \ar@{->}[rd]_{\rho} \\
& a^3
 }
\]
i.e, this ambiguity is resolvable. It is clear that we have no any other ambiguity and thus by the Composition--Diamond lemma the polynomials $bab-a$, $ba^2 - a^2b$ form a Gr\"obner--Shirshov basis of the ideal $R=(bab -a)$ relative to the order $\le$. Hence, by the Composition--Diamond lemma, $\mathrm{Irr}(R)$ is a $\mathbb{K}$-basis for the semigroup algebra $\mathbb{K}[K^+]$, therefore the basis elements of $K^+$ are all words $w$ of the $W$ which does not contain subwords of form $bab$ and $ba^2$. It is easy to verify that
\[
 \mathrm{Irr}(R) = \bigcup_{n,m \ge 1} \{1,a^n,b^m,a^nb^m, b^ma, a^nb^ma\}.
\]

For any $m \ge 0$ let us consider the following subset $\mathfrak{B}_m$ of $\mathrm{Irr}(R)$
\[
 \mathfrak{B}_m = \bigcup_{k \ge 0} \{ 1, a,b^k,b^ka, ab^{m+k}, ab^{m+k}a \}.
\]

Let us consider the corresponding germ $\Upsilon_m: = \Upsilon(K^+,\mathfrak{B}_m)$ where the partial map $\bullet:\Upsilon_m^{[2]} \to \Upsilon_m$ is defined as follows

\begin{center}
 \begin{tabular}{l||l|l|l|l|l|l}
 $\bullet$ & $1$ & $a$ & ${b^k}$ & ${b^na}$ & ${ab^{m+p}}$ & ${ab^{m+q}a}$\\
 \hline
 \hline
 $1$ & $1$ & $a$ & ${b^k}$ & ${b^na}$ & ${ab^{m+p}}$ & ${ab^{m+q}a}$ \\ 
 \hline
 $a$ & $a$ &  & $a \bullet {b^k}$ & $a\bullet {b^na}$ &  & \\ 
 \hline
 ${b^r}$ & ${b^r}$ & ${b^ra}$ & ${b^{r+k}}$ & ${b^{r+n}a}$ & ${b^r} \bullet {ab^{m+p}}$ & ${b^r} \bullet {ab^{m+q}a}$\\ \hline
 ${b^sa}$ & ${b^sa}$ &  & ${b^sa}\bullet {b^k}$ & ${b^sa}\bullet {b^na}$ &  & \\ 
 \hline
 ${ab^{m+t}}$ & ${ab^{m+t}}$ & ${ab^{m+t}a}$ & ${ab^{m+t+k}}$ & ${ab^{m+t+n}a}$ & ${ab^{m+t}}\bullet {ab^{m+p}}$ & \\ 
 \hline
 ${ab^{m+l}a}$ & ${ab^{m+l}a}$ &  & ${ab^{m+l}a}\bullet {b^k}$ &  &  & \\ 
 \hline
 \end{tabular}
\end{center}
here

\begin{eqnarray*}
  a \bullet {b^k} &:=& {ab^k}, \quad \mbox{if $k \ge m$},\\
  a \bullet {b^na} &:=& {ab^na}, \quad \mbox{if $n \ge m$},\\
  {b^r} \bullet {ab^{m+p}} &:=& \begin{cases} a, & \mbox{if $r=m+p$,} \\ {b^{r-m-p}a}, & \mbox{if $r > m+p$, $r-p \ge 2m$,} \\ {ab^{m+p-r}}, & \mbox{if $r<m+p$, $p \ge r$,}
  \end{cases}\\
  {b^r} \bullet {ab^{m+q}a} &:=& {ab^{q+m-r}a}, \quad \mbox{if $r<q+m$, $q\ge r$,}\\
  {b^sa} \bullet {b^k} &:=& \begin{cases}
    a, & \mbox{if $s = k$,} \\
    {b^{s-k}a},& \mbox{if $s>k$, $s-k \ge m$,}\\
    {ab^{k-s}}, &\mbox{if $s<k$, $k-s \ge m$,}
  \end{cases}\\
  {ab^{m+t}}\bullet {ab^{m+p}} &:=& {ab^{t-p}a}, \quad \mbox{if $t-p \ge m$},\\
 {ab^{m+p}} \bullet {ab^{k}} &:=&  ab^{m+p-k}a, \quad \mbox{if $m+p>k$, $p \ge k$,} 
\end{eqnarray*}

Set $S = \cup_{u,v \in \mathfrak{B}_m}\{(uv, u\bullet v)\}$. By the straightforward verification it is easy to see that any ambiguity of the set $S$ of reductions are resolvable with respect to the preorder $\le$. Hence, by Theorem \ref{The_Main_Result}, the germ $\Upsilon_m$ is Garside germ for any $m \ge 0,$ and hence, $\mathfrak{B}_m$ is a Garside family in $K^+$ for any $m\ge 0.$ Then, by (\ref{I(Ea,Eb)}), the corresponding $\mathcal{I}$-greatest function is defined as follows $\mathfrak{I}(u,v):=u\bullet v$ for all $u,v \in \mathfrak{B}_m$, it can be also identify with a function $H$ defined as $H(g):=\mathrm{min}_\preccurlyeq(g,ab^ma)$ (see \cite[A, IV, Example 2.35]{Dehbook}).
\end{example}

\section{Examples}

In this section we consider some interesting cases and examples.

\subsection{A Free Abelian Monoid}

We show how free abelian monoid can be obtained by using our method. We refer to \cite[3--5 pp.]{Dehbook} for details.

Take $n \ge 1$, and consider the free abelian monoid $\mathbb{N}^n$. It is clear that any its element $g$ can be viewed as a map $g: \{1,\ldots, n\} \to \mathbb{N}$. Denote by $g(k)$ the $k$th entry of $g.$

For $g, g' \in \mathbb{N}$, we define $g\cdot g'(k): = g(k) + g'(k)$ for each $k$. We define $\Delta_n$ by $\Delta_n(k) = 1$ for every $k$, and put $N_n:=\{\eta \in \mathbb{N}^n\, | \, \eta(k) \in \{0,1\} \mbox{ for any $k$}\}$. For $i\le n$, define $\alpha_i$ in $\mathbb{N}^n$ by $\alpha_i(k) = \delta_{i,j}$.

Thus we get
\[
 \mathbb{N}^n = \mathsf{Smg}\langle \alpha_i\alpha_j = \alpha_j\alpha_i, \mbox{ for all $1\le i,j \le n$} \rangle.
\]

For $f,g \in \mathbb{N}^n$, say that $f\lessapprox g$ is true if $f(i) \lessapprox g(i)$ holds for every $1\le i \le n$. We have (see \cite[A, I.1, Proposition 1.1]{Dehbook})

\textit{Every element of $\mathbb{N}^n$ admits a unique decomposition of the form $\Delta^d_n \eta_1\cdots \eta_p$ with $d$ in $\mathbb{N}$ and $\eta_1,\ldots, \eta_p \in N_n$ satisfying $\eta_1 \ne \Delta_n$, $\eta_p \ne 1$, and, for every $i<p$, $g \lessapprox \eta_{i+1}$ implies that $\eta_ig \not \lessapprox \Delta_n$, for any $g \in \mathbb{N}^n$.
}

For instance, let $n=3$,  and take $f = (5,4,3)$, we have
\begin{eqnarray*}
 (5,4,3) &=& (3,3,3) + (2,1,0) \\
 &=& (3,3,3) + (1,1,0) + (1,0,0),
\end{eqnarray*}
thus we obtain $f = \Delta_3^3\cdot (\alpha_1\alpha_2)\cdot \alpha_1$

Let us consider the following monoid (group) $M = (\mathbb{Z}/2)^n$, it can be also presented as follows
\[
 (\mathbb{Z}/2)^n = \mathsf{Sgr}\langle \theta, \xi_1,\ldots, \xi_n\, |\, \xi_i^2= \theta, \, 1 \le i \le n  \rangle,
\]
where $\theta(k) = 0 \bmod{2}$ for all $1\le k \le n$, $\xi_i(k) = \delta_{i,k}$, and for any $\varphi,\psi \in Z$, $\varphi \cdot \psi (k):= \varphi(k) + \psi(k) \bmod{2}$.

We have an epimorphism $P: \mathbb{N}^n \to (\mathbb{Z}/2)^n$, $f \mapsto f \bmod{2}$, where $(f\bmod{2})(k):=f(k)\bmod{2}$, for all $1\le k \le n.$

Let us consider, for instance, the following order; set $\xi_i > \xi_j$ if $i<j$, and consider then the corresponding deg-lex ordering on all elements of $(\mathbb{Z}/2)^n$. It is clear that the set of polynomials $\cup_{1 \le i\le n}\{\xi_i^2-\theta\}$ is a Gr\"obner--Shirshov basis. Hence, by the Composition--Diamond lemma (see Theorem \ref{CD-lemma} and Remark \ref{def_CD=CD}) the set of squarefree words in $\xi_i$ forms a basis $\mathfrak{B}$ for $(\mathbb{Z}/2)^n$. On the other hand, it is clear that any basic element is also an arbitrary map $\varphi:\{1,\ldots, n\} \to \mathbb{Z}/2$. Set $\varphi \cap \psi = \varnothing$ if $\varphi(k) \ne \psi(k)$ for all $1\le k \le n$, and $\varphi \cap \psi \ne \varnothing$ in otherwise.

Take a section $E:(\mathbb{Z}/2)^n \to \mathbb{N}$ of $P$ defined by $E(\varphi)(k) := E_\varphi(k):=\varphi(k \bmod{2})$, $1 \le k \le n$. By Constructions \ref{the_main_construction}, \ref{reduction_system}, we thus have a germ
\begin{eqnarray*}
 \Upsilon &=& \Upsilon_E(\mathbb{N},\mathbb{Z}/2,P,\mathfrak{B})\\
 &=&  \Upsilon = \left\{ E_\varphi, \, \varphi \in (\mathbb{Z}/2)^n\, |\, E_\varphi \bullet E_\psi = E_{\varphi \cdot \psi} \mbox{ whevere $\varphi \cap \psi = \varnothing$} \right\} 
\end{eqnarray*}
and a set of reductions
\[
\overline{S}_\Upsilon(\mathfrak{B}) = \bigcup_{\substack{\varphi \cap \psi\cdot \chi = \varnothing \\ \varphi \cdot \psi \cap \chi \ne \varnothing }}\left\{ E_\varphi E_{\psi\cdot \chi} \to E_{\varphi \cdot \psi} E_{\chi}\right\}.
\]

Next, we have $E_\varphi \gtrapprox E_\psi$ if there exists  $\varphi' \in (\mathbb{Z}/2)^n$ such that $\psi = \varphi\cdot \varphi'$ and $\varphi \cap \varphi' = \varnothing$. It is easy to see that $\Upsilon^\times = \{E_\theta\}.$ Since the number of all generators of $\Upsilon$ is finite then the preorder $\gtrapprox$ has descending chain condition.

By the straightforward computation one can get that all ambiguities of $\overline{S}_\Upsilon(\mathfrak{B})$ are resolvable, hence by Theorem \ref{CD-lemma}, the corresponding set of polynomials
\[
 (\overline{S}_\Upsilon(\mathfrak{B})) = \cup_{\substack{\varphi \cap \psi\cdot \chi = \varnothing \\ \varphi \cdot \psi \cap \chi \ne \varnothing }}\left\{ E_\varphi E_{\psi\cdot \chi} - E_{\varphi \cdot \psi} E_{\chi}\right\}
\]
is a Gr\"onber--Shirshov basis relative to $\lessapprox$. Therefore, by Theorem \ref{The_Main_Result}, $\Upsilon$ is a Garside germ and $\Upsilon$-normal form is the Gr\"obner--Shirshov normal form.

Finally, it is clear that the family $\cup_{\varphi \in (\mathbb{Z}/2)^n}\{E_\varphi\}$ is a generating family for $\mathbb{N}^n$ and hence by Theorem \ref{The_Main_Result}, 
\[
 \mathbb{N}^n \cong \mathsf{Smg} \langle E_\varphi, \, \varphi \in (\mathbb{Z}/2)^n\, |\, E_\varphi E_{\psi\cdot \chi} = E_{\varphi \cdot \psi}E_\chi \mbox{ whenever $\varphi \cap \psi\cdot \chi = \varnothing$ and $\varphi \cdot \psi \cap \chi \ne \varnothing $} \rangle,
\]
and $\cup_{\varphi \in (\mathbb{Z}/2)^n}\{E_\varphi\}$ is a Garside family in it.

\subsection{``A Natural Appearing of Invertible Elements''}\label{ex_of_empty_GSB}
Let us consider the following monoid
\[
 M = \mathsf{Smg} \langle p,p',r,q \, |\, prp'=p  \rangle.
\]

Using an arbitrary order on generators and consider the corresponding deg-lex order on all elements of the free monoid $W = W(p,p',r,q)$ generated by $p,p',r,q.$ We see that $\{prp'-p\}$ is a Gr\"obner--Shirshov basis for the ideal $(prp'-p)$ in $\mathbb{K}\langle p,p',r\rangle$. Hence by the classical version of the Composition--Diamond lemma (see Theorem \ref{CD-lemma}, Remark \ref{def_CD=CD}), 
\[
\mathfrak{B} = \bigcup_{u,v,w \in F}\{w \ne u prp'v\}
\]
is a basis for $M$. 

Next, let $\widetilde{M}$ be a left-cancellative monoid such that there is a surjective homomorphism $P:\widetilde{M} \to M$.

Let us consider a section $E:M \to \widetilde{M}$ of the $P$ and set
\begin{eqnarray*}
 \Upsilon &:=& \Upsilon_E(\widetilde{M},M,P,\mathfrak{B}) \\
  &=& \left\{ E_p,E_r,E_{p'},E_q\, E_{pr},E_{rq}\,| \, E_p\bullet E_r = E_{pr}, E_r \bullet E_q = E_{rq}, E_{pr} \bullet E_{p'} = E_p \right\}.
\end{eqnarray*}
and
\[
 \overline{S}_\Upsilon(\mathfrak{B}) = \{E_p E_r \to E_{pr}, E_r E_q \to E_{rq}, E_{pr}E_{p'} \to E_p\}. 
\]

Suppose further that $\Upsilon$ is a left-associative and left-cancellative germ. By $E_{pr}\bullet E_{p'} \in \Upsilon$, $E_r \bullet E_{p'} \in \Upsilon$, and by $\Upsilon \ni E_p\bullet E_{r} = (E_{pr}\bullet E_{p'})\bullet E_r$, $E_{p'}\bullet E_r \in \Upsilon.$ Therefore, by Definition \ref{def_of_germ} (3), $E_p \bullet (E_r \bullet E_{p'}) = (E_p \bullet E_r)\bullet E_{p'}$. Hence, $E_rE_{p'} = \mathbf{1}$ because of $(E_p \bullet E_r)\bullet E_{p'} = E_p \bullet (E_r \bullet E_{p'}) =  E_p$.

Next, let us consider an ambiguity $(E_r, E_{p'},E_{r})$, we obtain
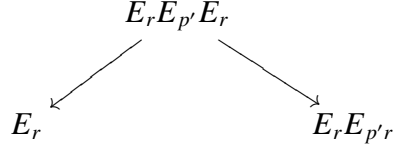
\[
 \xymatrix{
  & E_r E_{p'} E_r \ar@{->}[ld] \ar@{->}[rd] & \\
  E_r && E_{r}E_{p'r}
 }
\]
then by left-cancellativity of $\Upsilon$, $E_{p'}E_r = \mathbf{1}$. The ambiguity $(E_{p'},E_r,E_q)$ gives
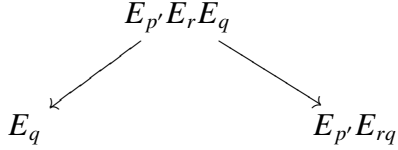
\[
 \xymatrix{
  & E_{p'} E_r E_q \ar@{->}[ld] \ar@{->}[rd] & \\
  E_q && E_{p'}E_{rq}
 }
\]
\textit{i.e.,} we have $E_{p'}E_{rq} = E_q$.

Finally let us consider an ambiguity $(E_p,E_r, E_q)$, we have
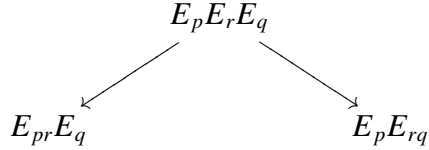
\[
  \xymatrix{
  & E_p E_{r} E_q \ar@{->}[ld] \ar@{->}[rd] & \\
  E_{pr}E_q && E_{p}E_{rq}
 }
\]
but $E_{pr} \preccurlyeq E_p$ because of $E_p = E_{pr}E_{p'}$, \textit{i.e.,} $E_p \not \gtrapprox E_{pr}$. Thus if $\thickapprox$ is a deformation of elements by invertible elements (see Definition \ref{def_of_deformation}), we then get that all ambiguities of the system $\overline{S}_\Upsilon(\mathfrak{B})$ are resolvable up to $\thickapprox$ because of
\[
 \xymatrix{
  \cdot \ar@{->}[r]^{E_{p}} \ar@{=}[d] & \cdot \ar@{->}[r]^{E_{rq}} \ar@/_/[d]_{E_r} & \cdot \\
 \cdot \ar@{->}[r]_{E_{pr}} & \cdot \ar@{->}[r]_{E_{q}} \ar@/_/[u]_{E_{p'}} & \cdot \ar@{=}[u]
 }
\]
\textit{i.e.,} $E_{pr}E_q \thickapprox E_pE_{rq}$. Hence, by Theorem \ref{The_Main_Result}, $\Upsilon$ is a Garside germ.

\subsection{The Artin-Tits Monoid}

\subsubsection{Coxeter Systems and Groups}

Let $\mathcal{S}$ be a set. A matrix $\mathsf{M}_\mathcal{S}: \mathcal{S}\times \mathcal{S} \to \{1,2,\ldots, \infty\}$ is called {\it a Coxeter matrix } if it satisfies
\begin{align*}
  &\mathsf{m} (a,b) = \mathsf{m} (b,a);\\
  &\mathsf{m}(a,b) = 1 \mbox{ if and only if } a =b,
\end{align*}
here $a,b \in \mathcal{S}$.

Equivalently, $\mathsf{M}_\mathcal{S}$ can be represented by a {\it Coxeter graph} whose node set is $\mathcal{S}$ and whose edges are the unordered pairs $\{a,b\}$ such that $\mathsf{m}(a,b) \ge 3$. The edges with $\mathsf{m}(a,b) \ge 4$ are labeled by that numbers.

\begin{figure}[h!]
 \begin{tikzpicture}
  \draw (-3,0) node {$\begin{pmatrix}
    1 & 3 & 2 & 2\\
    3 & 1 & 3 & 3\\
    2 & 3 & 1 & 3\\
    2 & 2 & 3 & 1
  \end{pmatrix}$};
 \draw[line width =1] (0,0) -- (1,0);
  \draw[line width =1] (1,0) -- (2,0);
  \draw[line width =1] (2,0) -- (3,0);
  \draw[line width=1] (0,0) circle (3pt) node[below] {$s_1$};
  \draw[line width=1] (1,0) circle (3pt) node[below] {$s_2$};
  \draw[line width=1] (2,0) circle (3pt) node[below] {$s_3$};
  \draw[line width=1] (3,0) circle (3pt) node[below] {$s_4$};
  \fill[white] (0,0) circle(3pt);
  \fill[white] (1,0) circle(3pt);
  \fill[white] (2,0) circle(3pt);
  \fill[white] (3,0) circle(3pt);
  \draw (-3,-3) node {$\begin{pmatrix}
    1 & 3 & 3 & 2\\
    3 & 1 & 5 & 2\\
    3 & 5 & 1 & \infty\\
    2 & 2 & \infty & 1
  \end{pmatrix}$};
  \draw[line width =1] (0,-2.3) -- (0,-3.7);
  \draw[line width =1] (0,-3.7) -- node[below]{$5$}(1.5,-3);
  \draw[line width =1] (0,-2.3) -- (1.5, -3);
  \draw[line width =1] (1.5,-3) -- node[above]{$\infty$}(3,-3);
  \draw[line width=1] (0,-2.3) circle (3pt) node[above] {$s_1$};
  \draw[line width=1] (0,-3.7) circle (3pt) node[below] {$s_2$};
  \draw[line width=1] (1.5,-3) circle (3pt) node[below] {$s_3$};
  \draw[line width=1] (3,-3) circle (3pt) node[below] {$s_4$};
  \fill[white] (0,-2.3) circle(3pt);
  \fill[white] (0,-3.7) circle(3pt);
  \fill[white] (1.5,-3) circle(3pt);
  \fill[white] (3,-3) circle(3pt);
 \end{tikzpicture}
 \caption{The Coxeter matrix and its corresponding graph.}
\end{figure}
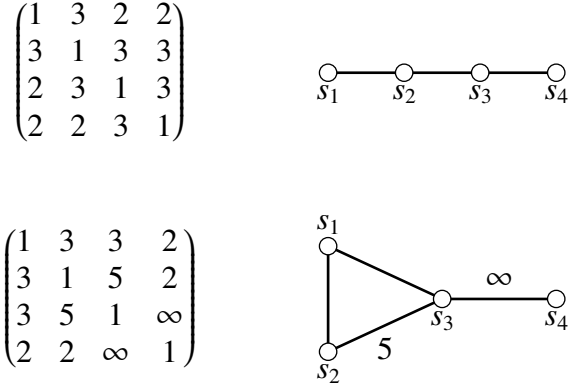

\begin{definition}
 Let $\mathcal{S}$ be a set, $\mathcal{W}$ a group. We say that $\mathcal{(W,S)}$ is a \textit{Coxeter system}, and that $\mathcal{W}$ is a \textit{Coxeter group}, if $\mathcal{W}$ admits the presentation
 \[
  \mathcal{W} = \mathrm{Smg} \langle \mathcal{S} \, |\,  s^2 = 1,\, \mbox{for all $s \in \mathcal{S}$, and whenever $\mathsf{m}(a,b) \ne \infty$, $(ab)^{\mathsf{m}(a,b)} = (ba)^{\mathsf{m}(b,a)}$}  \rangle.
 \]
\end{definition}

Given a Coxeter system $\mathcal{(W,S)}$, each element $w \in \mathcal{W}$ can be written as a product of generators: $w = s_1s_2\cdots s_k$, $s_i \in \mathcal{S}$. If $k$ is minimal among all such expressions for $w$, then $k$ is called the \textit{length} of $w$ (written $\ell(w) =k$) and the word $s_1s_2\cdots s_k$ is called a \textit{reduced word} (or \textit{reduced decomposition} or \textit{reduced expression}) for $w$.

The following properties are fundamental in the combinatorial theory of Coxeter groups: they characterize such groups.

\textbf{(Deletion Property).} \textit{If $w = s_1s_2\cdots s_k$ and $\ell(w)<k$, then $w = s_1\cdots \widehat{s_i} \cdots \widehat{s_j} \cdots s_k$ for some $1\le i<j\le k.$
}

\textbf{(Exchange Property).} \textit{Let $w = s_1s_2\cdots s_k$ be reduced expression and $s\in \mathcal{S}$. If $\ell(sw)\le \ell(w)$, then $sw = s_1\cdots \widehat{s_i}  \cdots s_k$ for some $1\le i<j\le k,$ for some $1 \le i \le k.$
}

\begin{theorem}[{\cite[Theorem 1.5.1]{BB}}]
 Let $\mathcal{W}$ be a group and $\mathcal{S}$ a set of generators of order $2$. Then the following are equivalent.
 \begin{enumerate}
     \item $\mathcal{(W,S)}$ is a Coxeter system.
     \item $\mathcal{(W,S)}$ has the Exchange Property.
     \item $\mathcal{(W,S)}$ has the Deletion Property.
 \end{enumerate}
\end{theorem}

For given $u,v \in \mathcal{W}$ we write $u \perp v$ if $\ell(uv) = \ell(u) + \ell(v)$, and $u \angle v$ if $\ell(uv) \ne \ell(u) + \ell(v).$

\begin{lemma}\label{perp}
   Assume that $\ell(uv) = \ell(u) + \ell(v)$ and $\ell(vw) = \ell(v) + \ell(w)$ for some reduced $u,v,w \in W$, \textit{i.e.,} $u\perp v$ and $v \perp w$. Then $uv \perp v$ and $u \perp vw.$
\end{lemma}

\begin{proof} Let $u = s_{i_1}\cdots s_{i_p}$, $v = s_{j_1}\cdots s_{j_q}$, and $w = s_{k_1}\cdots s_{k_r}$.

Let us assume that $\ell((uv)w) < \ell(uv) + \ell(w)$ or $\ell(u(vw))< \ell(u) + \ell(vw)$. Then by the Deletion Property, $uvw = s_{i_1}\cdots \widehat{s_\alpha} \cdots \widehat{s_{\beta}} \cdots s_{k_r}$. Since $\ell(uv) = \ell(u) + \ell(v)$ and $\ell(vw) = \ell(v) + \ell(w)$ then 1) either $\alpha \in \{i_1,\ldots, i_p\}$, $\beta \in \{j_1,\cdots, j_q\}$, or 2) $\alpha \in \{i_1,\ldots, i_p\}$, $\beta \in \{k_1,\cdots, k_r\}$, or 3) $\alpha \in \{j_1,\ldots, j_q\}$, $\beta \in \{k_1,\cdots, k_r\}$. On the other hand, $u,v,w$ are assumed to be reduced, hence we get a contradiction and the statement follows. 
\end{proof}

\begin{corollary}\label{perp2}
 If $uv \perp w$ and $u \perp v$, then $u \perp vw$ and $v \perp w$.
\end{corollary}
\begin{proof}
 By $uv \perp w$, and $u \perp v$, $\ell(uvw) = \ell(u) + \ell(v) + \ell(w)$, and using Lemma \ref{perp}, the statement follows.
\end{proof}

\begin{definition}[{\cite[3.4]{BB}}]
 Let $(\mathcal{W,S})$ be a Coxeter system, $\mathfrak{R}(w)$ be the set of all reduced decompositions of an element $w$. The \textit{normal form} of an element $w \in \mathcal{W}$ is $\mathrm{min} \mathfrak{R}(w)$, where the minimal is taken with respect to lexicographic order. Denote by $\mathfrak{B}(\mathcal{W})$ the set of all elements $w\in \mathcal{W}$ such that $\mathrm{min}\,\mathfrak{R}(w) = w.$
\end{definition}

As before we denote by $\mathrm{NF}(w)$ a normal form of an element $w.$ We refer to \cite{BB} (especially section 3.4) for details and how to compute a normal form of an element of an arbitrary Coxeter group. However, just for reader convenience, we recall an elegant algorithm, see Appendix, to compute a normal form for any Coxeter group. This algorithm is based on so-called ``the numbers game''  (see \cite[4.3]{BB}) offers a general method for finding combinatorial representatives of the group elements. 

\subsubsection{Artin--Tits monoids}

\begin{definition}
Let $(\mathcal{W,S})$ be a Coxeter system given by a Coxeter matrix $\mathsf{M}$, an {\it Artin-Tits monoid} associated with $(\mathcal{W,S})$, is a monoid, denoted by $B^+(\mathcal{W,S})$ or shortly $B^+$, admits the presentation
\[
 B^+(\mathcal{W,S}):=\left< \mathcal{S}\, \bigl|\, \langle a ,b\rangle^{\mathsf{m}(a,b)} = \langle b, a \rangle^{\mathsf{m}(b,a)} \mbox{ whenever $\mathsf{m}(a,b) \ne \infty$} \bigr.\right>
\]
here $\mathsf{m}(a,b)$ are elements of the Coxeter matrix $\mathsf{M}$. If $\mathsf{m}(a,b)=\infty$, then there is no relation for $a$ and $b$.
\end{definition}

Let us turn to the preorder $\gtrapprox$ (Definition \ref{order}). It is clear the $E_{\mathrm{NF}(u)} \gtrapprox E_{\mathrm{NF}(v)}$ if and only if $\ell(u) < \ell(v)$.

\begin{proposition}[{cf. \cite[B, IX.1.3, Proposition 1.35 ]{Dehbook}}]\label{artin-GSB}
 Let $(\mathcal{W,S})$ be a Coxeter system. Consider the corresponding Artin-Tits monoid $B^+(\mathcal{W,S})$. Let $P: B^+ \to \mathcal{W}$ be the corresponding surjective, and $E:\mathcal{W} \to B^+$ its sections which is defined as follows: any $s \in \mathcal{S}$ maps to the same $s$ in $B^+.$ Define the corresponding germ $\Upsilon = \Upsilon_E(B^+,\mathcal{W},P,\mathfrak{B}(\mathcal{W}))$ as follows 
\[
 \Upsilon \coloneqq \mathsf{Smg} \bigl \langle E_{\mathrm{NF}(w)},\, w \in \mathcal{W}\mid E_{\mathrm{NF}(u)} \bullet E_{\mathrm{NF}(v)}\coloneqq  E_{\mathrm{NF}(uv)}, \mbox{whenever $u \perp v$}  \bigr\rangle.
\]

Then the germ $\Upsilon$ is a Garside germ, and
\[
 B^+(\mathcal{W,S}) \cong \mathsf{Smg} \langle E_{\mathrm{NF}(w)},\, w \in \mathcal{W}\, |\, E_{\mathrm{NF}(u)}E_{\mathrm{NF}(vw)} = E_{\mathrm{NF}(uv)}E_{\mathrm{NF}(w)},\, u\perp v \perp w  \rangle,
\]
and finally a basis of $B^+(\mathcal{W,S})$ can be described as the following set
\[
 \bigcup_{\substack{u, v \ldots, w \in \mathcal{W} \\ u \angle v \angle \cdots \angle w }} \{ E_{\mathrm{NF}(\pi_1)}E_{\mathrm{NF}(\pi_2)} \cdots E_{\mathrm{NF}(\pi_k)}\}.
\]
\end{proposition}
\begin{proof}
Let us prove that the germ $\Upsilon$ is a Garside germ.

First of all we have to show that $\Upsilon$ is a left-cancellative and left-associative germ.
 
(1) Let $E_u \bullet E_v, E_v \bullet E_w \in \Upsilon$, then $u\perp v$ and $v \perp w$. By Lemma \ref{perp}, $E_u\bullet (E_v \bullet E_w)\in \Upsilon$ if and only if $(E_u \bullet E_v) \bullet E_w \in \Upsilon$, and if so then they are equal.

(2) Since $\mathcal{W}$ is assumed to be a group then $\Upsilon$ is left and right-cancellative. Next, by Corollary \ref{perp2}, $\Upsilon$ is left-associative.

(3) By Construction \ref{reduction_system} we have the following set of reductions 
\[
 \overline{S}_\Upsilon(\mathfrak{B}(\mathcal{W})) = \bigcup_{u\perp v \perp w}\{\mathfrak{r}_{u,vw}: E_{\mathrm{NF}(u)}E_{\mathrm{NF}(vw)} \to E_{\mathrm{NF}(uv)}E_{\mathrm{NF}(w)}\}.
\]
 It is clear that all possible ambiguities are $(\mathfrak{r}_{u,vw}, \mathfrak{r}_{vw,hg}, E_{\mathrm{NF}(u)},E_{\mathrm{NF}(v w)}, E_{\mathrm{NF}(h g)})$ where $u \perp v \perp w$ and $v w \perp h \perp g$. 

By Corollary \ref{perp2}, $w \perp h$, $v \perp w h$ and we get
\[
 \xymatrix{
 & E_{\mathrm{NF}(uv)}E_{w}E_{\mathrm{NF}(hg)} \ar@{->}[r]^{\mathfrak{r}_{w,hg}}& E_{\mathrm{NF}(uv)}E_{\mathrm{NF}(wh)}E_g  \\
 E_uE_{\mathrm{NF}(vw)}E_{\mathrm{NF}(hg)} \ar@{->}[ru]^{\mathfrak{r}_{u,vw}} \ar@{->}[rd]_{\mathfrak{r}_{vw,hg}} &&&\\
 & E_u E_{\mathrm{NF}(vwh)}E_g \ar@{->}[r]_{\mathfrak{r}_{uv,wh}} & E_{\mathrm{NF}(uv)}E_{\mathrm{NF}(wh)}E_g
 }
\]

Thus all ambiguities of $\overline{S}_\Upsilon(\mathfrak{B}(\mathcal{W}))$ are resolvable and hence by Theorem \ref{The_Main_Result} the first statement follows. 

Next, it is easy to see that the set $\cup_{w \in \mathcal{W}}\{E_{\mathrm{NF}(w)}\}$ is a generating family for $B^+(\mathcal{W,S})$. Hence, by Proposition \ref{Garside_germ=cat}, the second statement follows.

Finally, by Theorem \ref{The_Main_Result}, the last statement follows.
\end{proof}

It is clear that if a Coxeter group $\mathcal{W}$ is finite then the corresponding preorder $\gtrapprox$ has the descending chain condition, thus, by Corollary \ref{GSB=greedy}, the corresponding $\Upsilon$-normal form of an Artin monoid is exactly the corresponding Gr\"obner--Shirshov normal form.

\subsection{A Greedy normal form on the braid monoids $B_n^+$}~\\

Let us consider the partial case $\mathcal{W} = \mathfrak{S}_n$ (= the symmetric group). It is well known that a symmetric group has the following Coxeter presentation.

\begin{proposition}[{\cite{BS01,BBCM}}]
Let $\mathcal{S} = \{s_1, \ldots, s_{n-1}\}$ be the set of generators (transpositions) of the symmetric group $\mathfrak {S}_n$. Set $s_i >s_j$ whenever $i>j$ and consider the corresponding deg-lex ordering $>$ on the free monoid generated by $s_1,\ldots, s_{n-1}$. 

A Gr\"obner--Shirshov basis for the symmetric group $\mathfrak{S}_{n}$, with respect to the order $>$, is the following set of relations:
\begin{enumerate}
    \item $s_i^2 = 1,$ for any $1 \le i \le n$ 
    \item $s_is_j = s_js_i,$ for $i-j \ge 2$ and $1 \le i,j \le n,$
    \item $s_{i+1}s_is_{i-1} \cdots s_js_{i+1} = s_is_{i+1}s_is_{i-1}\cdots s_j,$ if $i+1 \ge j$ and $1\le i,j \le n$.
\end{enumerate}
\end{proposition}

As a consequence, using the Composition--Diamond lemma, we obtain the following

\begin{corollary}[\cite{BS01}]
 The set 
 $\mathfrak{B}(\mathfrak {S}_n)\coloneqq \{s_{1i_1}s_{2i_2}\cdots s_{ni_n} \mid
 i_k \le k+1\}$ 
 consists of Gr\"obner--Shirshov normal forms for $\mathfrak {S}_n$ in the generators 
$s_i\coloneqq (i,i+1)$ relative to the deg-lex ordering, where 
$s_{\alpha \beta}\coloneqq s_\beta s_{\beta-1}\cdots s_\alpha$ 
for 
$\beta \ge \alpha$ 
and 
$s_{\beta, \beta+1}\coloneqq 1$.
\end{corollary}

Let us consider a braid monoid $B_n^{+}$, \textit{i.e.,} a monoid generated by $\sigma_1,\ldots, \sigma_n$; its elements are called positive braids. We have a homomorphism $P: B_n^+ \to \mathfrak{S}_n$; given a positive braid $B$, the strands define a permutation $p(B)$ from the top set of endpoints to the bottom set of endpoints.

Take $\pi \in \mathfrak {S}_n$ with the normal form $\mathrm{NF}(\pi) = s_{1i_1}s_{2i_2}\cdots s_{mi_m} \in \mathrm{NF}(\mathfrak{S}_n)$, \textit{i.e.,} 
\[
 \mathrm{NF}(\pi) = (s_{i_1}s_{i_1-1}\cdots s_1) (s_{i_2}s_{i_2-1}\cdots s_2) \cdots (s_{i_m}s_{i_m-1}\cdots s_{m})
\]
and, as above, set $\ell(\pi) :=  \ell(\mathrm{NF}(\pi))$ (= the length). Define then a map $E:\mathfrak{S}_n \to B_n^+$ as follows
\[
 E(\mathrm{NF}(\pi)) = (\sigma_{i_1}\sigma_{i_1-1}\cdots \sigma_1) (\sigma_{i_2}\sigma_{i_2-1}\cdots \sigma_2) \cdots (\sigma_{i_m}\sigma_{i_m-1}\cdots \sigma_{m}),
\]
it is clear that $E$ is a section for $P$.

Next, as above, we see that the corresponding germ $\Upsilon=\Upsilon_E(B_n^+,\mathfrak{S}_n,P,\mathfrak{B}(\mathfrak{S}_n))$ can be also described as follows
\[
 \Upsilon = \bigl\{ E_{\mathrm{NF}(\pi)},\, \pi \in \mathfrak {S}_n  \mid E_{\mathrm{NF}(\pi)} \bullet E_{\mathrm{NF}(\tau)}:= E_{\mathrm{NF}(\pi \tau)}, \mbox{ whenever $\pi \perp \tau$} \bigr\}.
\]

Thus, by Proposition \ref{artin-GSB}, \textit{the corresponding set of reductions 
\[
\overline{S}_\Upsilon(\mathfrak{B}(\mathfrak{S}_n)) = 
\bigcup_{\pi,\tau, \xi \in \mathfrak {S}_n}\{ \mathfrak{r}_{\pi,\tau\xi}: E_{\mathrm{NF}(\pi)} E_{\mathrm{NF}(\tau \xi)}\to E_{\mathrm{NF}(\pi \tau)}E_{\mathrm{NF}(\xi)}, \, \pi \perp \tau \perp \xi\},
\]
is resolvable and, by Theorem \ref{CD-lemma}, the corresponding set of polynomials
\[
(\overline{S}_\Upsilon(\mathfrak{B}(\mathfrak{S}_n)))= \bigcup_{\pi,\tau, \xi \in \mathfrak {S}_n}\{ \mathfrak{r}_{\pi,\tau\xi}: E_{\mathrm{NF}(\pi)} E_{\mathrm{NF}(\tau \xi)} - E_{\mathrm{NF}(\pi \tau)}E_{\mathrm{NF}(\xi)}, \, \pi \perp \tau \perp \xi\},
\]
is a Gr\"obner--Shirshov basis relative to the $\gtrapprox$, and then the corresponding Gr\"obner--Shirshov normal form is
normal form is exactly the greedy normal form, \textit{i.e.,} the set $\mathrm{Irr}(\overline{S}_\Upsilon(\mathfrak{B}(\mathfrak{S}_n)))$ of irreducible elements under the $\overline{S}_\Upsilon(\mathfrak{B}(\mathfrak{S}_n))$
\[
 \mathrm{Irr}(\overline{S}_\Upsilon(\mathfrak{B}(\mathfrak{S}_n))):=\bigcup_{\substack{\pi_1, \ldots, \pi_k \in \mathfrak{S}_n \\ \pi_1 \angle \pi_2 \angle \cdots \angle \pi_k }} \{ E_{\mathrm{NF}(\pi_1)}E_{\mathrm{NF}(\pi_2)} \cdots E_{\mathrm{NF}(\pi_k)}\},
\]
coincides with the set of the greedy normal form of elements of $B_n^+$. Finally for any $\pi \in \mathfrak{S}_n$, the element $E_{\mathrm{NF}(\pi)}$ coincides with the Adjan--Thurston generator $R_\pi.$
}

\begin{remark}
In \cite[A, VI, Example 2.72]{Dehbook} it was shown how to obtain 
the braid monoid (and greedy normal form) via the symmetric group. This approach is very similar to this way. In particular the elements of $\pi,\tau \in \mathfrak {S}_n$ are called {\em tight}\index{Group!braid!tight elements} if $\ell(\pi \tau) = \ell(\pi) + \ell(\tau)$ (in a sense described there).
\end{remark}

\begin{remark}
In \cite[3.1.4]{BokSurv} it was said that the corresponding set of polynomials
\[
(\overline{S}_\Upsilon(\mathfrak{B}(\mathfrak{S}_n))) = \cup_{\pi \perp \tau \perp \xi}\{E_{\mathrm{NF}(\pi)} E_{\mathrm{NF}(\tau \xi)}- E_{\mathrm{NF}(\pi \tau)}E_{\mathrm{NF}(\xi)}\}
\]
is a Gr\"obner--Shirshov basis relative to the following ordering $\le$. We assume that $s_1<s_2<\cdots <s_{n-1}$ and define $E_{\mathrm{NF}(\pi)} \le E_{\mathrm{NF}(\tau)}$ if and only if $\ell(\mathrm{NF}(\pi)) > \ell(\mathrm{NF}(\tau))$ or $\ell(\mathrm{NF}(\pi)) = \ell(\mathrm{NF}(\tau))$ and $\mathrm{NF}(\pi) < \mathrm{NF}(\tau)$ (lexicographical order). It is easy to see that this order is an extension of the preorder $\gtrapprox.$
\end{remark}

\section*{Appendix: The Number Game}

Let us recall the ``number game'' that first appeared in a somewhat restricted version related to Kac--Moody Lie algebras in \cite{M} and the general version given in \cite{BB} that is due to Eriksson \cite{E}. We refer to \cite[Ch.I,4]{BB} for more details.

Let $(\mathcal{W,S})$ be a Coxeter system given by a Coxeter matrix $\mathsf{M}$. Define a function $\kappa:\mathcal{ S \times S} \to \mathbb{R}$ as follows;
\begin{equation}\label{k}
 \begin{cases}
   \begin{cases}
     \kappa_{s,s} = -2, & s \in \mathcal{S};\\
     \kappa_{a,b} = 0, & \mathsf{m}(a,b) = 2,
   \end{cases}, & \mathsf{m}(-,=) <3,\\
   \begin{cases}
     \kappa_{a,b} >0,\\
     \kappa_{a,b}\kappa_{b,a} = 4 \cos^2 \dfrac{\pi}{\mathsf{m}(a,b)}, & \mathsf{m}(a,b) \ne \infty,\\
     \kappa_{a,b}\kappa_{b,a} \ge 4, & \mathsf{m}(a,b) = \infty
   \end{cases}, & \mathsf{m}(-,=) \ge 3
 \end{cases}
\end{equation}

Now we present (see \cite[I.4.3]{BB}) the Number Game to calculate normal form. We label each node of the corresponding Coxeter graph with some real numbers, and each such assignment thought of as a position in a certain ``game''. The ``moves'' in the game are local rearrangements of the assigned values at a chosen node $s$ and its neighbors, governed by the labels of the edges surrounding in $s$ in the Coxeter graph. The point of this game is that it gives a combinatorial model of the Coxeter group, where group elements correspond to positions and reduced decompositions correspond to play sequences.

The starting position for the game can be any distribution $\mathcal{S} \ni s \to p_s \in \mathbb{R}$ or real numbers $p_s$ to the nodes $s\in \mathcal{S}$ of the Coxeter graph. A position is called positive if $p_s>0$ for all $s \in \mathcal{S}$. The special position with $p_s=1$ for all $s \in \mathcal{S}$ is called the unit position and denoted by $\mathbf{1}.$

Next, moves are defined as follows. A \textit{firing of node} $s$ changes a position $p \in \mathbb{R}^\mathcal{S}$ in the following way
\begin{enumerate}
    \item Switch sign of the value at $s$.
    \item Add $k_{s,a}p_a$ to the value at each neighbor $a$ of $s$.
    \item Leave all other values unchanged.
\end{enumerate}

Such a move is called \textit{positive} if $p_s>0$, and \textit{negative} if $p_s<0$. A \textit{positive game} is one that is played with positive moves from a given starting position, and similarly for a negative game. A \textit{play sequence} is a word $s_{i_1}s_{i_2}\cdots s_{i_k}$, $s_{i_j} \in \mathcal{S}$, recording a game in which $s_{i_1}$ was fired first, then $s_{i_2}$, then $s_{i_3}$, and so on. Similarly, a \textit{positive play sequence} records a positive game and a \textit{negative play sequence} records a negative game.

\begin{theorem}[{\cite[Theorem 4.3.1]{BB}}]
 Let $p \in \mathbb{R}^\mathcal{S}$ be a starting position, $s_{i_1},\ldots, s_{i_k}$ a play sequence, denote by $p^{s_{i_1}\cdots s_{i_k}}$ a position reached from $p$ by this play sequence. Let $\mathfrak{P}_p \subseteq \mathbb{R}^\mathcal{S}$ denote the set of all positions that can be reached this way.
 \begin{itemize}
     \item Two play sequences $s_{i_1}s_{i_2}\cdots s_{i_p}$ and $s_{j_1}s_{j_2}\cdots s_{j_q}$ lead to the same position if and only if $s_{i_1}s_{i_2}\cdots s_{i_p}=s_{j_1}s_{j_2}\cdots s_{j_q}$ as elements of $\mathcal{W}$.
     \item The induced mapping $w \mapsto p^w$ is a bijection $\mathcal{W} \to \mathfrak{P}_p$.
     \item The play sequence $s_{i_1}s_{i_2}\cdots s_{i_k}$ is positive if and only if $s_{i_1}s_{i_2}\cdots s_{i_k}$ is a reduced decomposition.
 \end{itemize}
\end{theorem}

This Theorem implies the following algorithm for finding normal form $\mathrm{NF}(w)$ of and element $w$ of a Coxeter group $\mathcal{W}$.

\begin{enumerate}
    \item Take an expression $w = s_{i_1}s_{i_2}\cdots s_{i_k}$.
    \item Play from $\mathbf{1}$ according to the play sequence $s_{i_k}, \ldots, s_{i_2}, s_{i_1}$.
    \item Set $p:=p^{s_{i_k}\cdots s_{i_2}s_{i_1}}$.
    \item Play from $p$ to $\mathbf{1}$ by firing at each step the minimal negative node.
    \item Record the obtained play sequence $s_{j_1},s_{j_2},\ldots, s_{j_n}.$
    \item $\mathrm{NF}(w)=s_{j_1},s_{j_2},\ldots, s_{j_n}.$
\end{enumerate}

\begin{example}
  Let us consider the following Coxeter graph
  
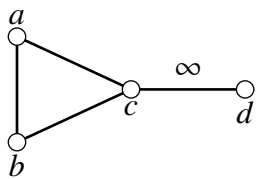
\begin{figure}[h!]
    \begin{tikzpicture}
      \draw[line width =1] (0,-2.3) -- (0,-3.7);
      \draw[line width =1] (0,-3.7) -- (1.5,-3);
      \draw[line width =1] (0,-2.3) -- (1.5, -3);
      \draw[line width =1] (1.5,-3) -- node[above]{$\infty$}(3,-3);
      \draw[line width=1] (0,-2.3) circle (3pt) node[above] {$a$};
      \draw[line width=1] (0,-3.7) circle (3pt) node[below] {$b$};
      \draw[line width=1] (1.5,-3) circle (3pt) node[below] {$c$};
      \draw[line width=1] (3,-3) circle (3pt) node[below] {$d$};
      \fill[white] (0,-2.3) circle(3pt);
      \fill[white] (0,-3.7) circle(3pt);
      \fill[white] (1.5,-3) circle(3pt);
      \fill[white] (3,-3) circle(3pt);
    \end{tikzpicture}
  \end{figure}

Put $a<b<c<d$. Let us find the normal form of the word $w = babcdb$. Let us find the starting position (see fig.\ref{SPE1}). We have $\m{p} =(-3,-6,-4,7)$. Now we have to play from this position to position $\mathbf{1}$ by firing at each step the minimal negative node (see fig.\ref{NFE1}). Thus we obtain $\mathrm{NF}(w) = abacbd$.
\end{example}

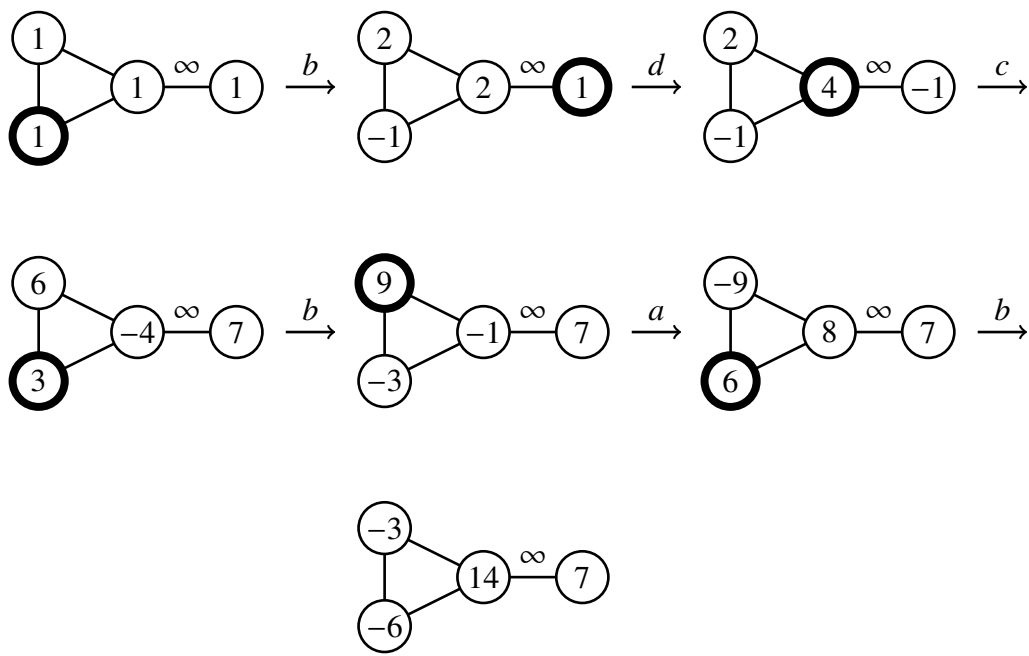
\begin{figure}[h!]
\begin{tikzpicture}[line width =1,scale =0.65]
  \node(a) at (0,0) {};
  \node(b) at (0,-2) {};
  \node(c) at (2,-1){};
  \node(d) at (4,-1) {};
  \draw (0,0)--(0,-2);
  \draw (0,0) --(2,-1);
  \draw (0,-2)--(2,-1);
  \draw (2,-1)--node[above]{$\infty$}(4,-1);
  \fill[white] (a) circle(15pt);
  \fill[white] (b) circle(15pt);
  \fill[white] (c) circle(15pt);
  \fill[white] (d) circle(15pt);
  \node(a) at (0,0) {$1$};
  \draw (a) circle(15pt);
  \node(b) at (0,-2) {$1$};
  \draw[line width = 3] (b) circle(15pt);
  \node(c) at (2,-1) {$1$};
  \draw (c) circle(15pt);
  \node(d) at (4,-1) {$1$};
  \draw (d) circle(15pt);
  \draw[->] (5,-1)--node[above]{$b$}(6,-1);
  \begin{scope}[xshift = 7cm]
  \node(a) at (0,0) {};
  \node(b) at (0,-2) {};
  \node(c) at (2,-1){};
  \node(d) at (4,-1) {};
  \draw (0,0)--(0,-2);
  \draw (0,0) --(2,-1);
  \draw (0,-2)--(2,-1);
  \draw (2,-1)--node[above]{$\infty$}(4,-1);
  \fill[white] (a) circle(15pt);
  \fill[white] (b) circle(15pt);
  \fill[white] (c) circle(15pt);
  \fill[white] (d) circle(15pt);
  \node(a) at (0,0) {$2$};
  \draw (a) circle(15pt);
  \node(b) at (0,-2) {$-1$};
  \draw (b) circle(15pt);
  \node(c) at (2,-1) {$2$};
  \draw (c) circle(15pt);
  \node(d) at (4,-1) {$1$};
  \draw[line width = 3] (d) circle(15pt);
  \draw[->] (5,-1)--node[above]{$d$}(6,-1);
   \end{scope}
\begin{scope}[xshift = 14cm]
  \node(a) at (0,0) {};
  \node(b) at (0,-2) {};
  \node(c) at (2,-1){};
  \node(d) at (4,-1) {};
  \draw (0,0)--(0,-2);
  \draw (0,0) --(2,-1);
  \draw (0,-2)--(2,-1);
  \draw (2,-1)--node[above]{$\infty$}(4,-1);
  \fill[white] (a) circle(15pt);
  \fill[white] (b) circle(15pt);
  \fill[white] (c) circle(15pt);
  \fill[white] (d) circle(15pt);
  \node(a) at (0,0) {$2$};
  \draw (a) circle(15pt);
  \node(b) at (0,-2) {$-1$};
  \draw (b) circle(15pt);
  \node(c) at (2,-1) {$4$};
  \draw[line width = 3] (c) circle(15pt);
  \node(d) at (4,-1) {$-1$};
  \draw (d) circle(15pt);
  \draw[->] (5,-1)--node[above]{$c$}(6,-1);
   \end{scope}
\begin{scope}[yshift = -5cm]
  \node(a) at (0,0) {};
  \node(b) at (0,-2) {};
  \node(c) at (2,-1){};
  \node(d) at (4,-1) {};
  \draw (0,0)--(0,-2);
  \draw (0,0) --(2,-1);
  \draw (0,-2)--(2,-1);
  \draw (2,-1)--node[above]{$\infty$}(4,-1);
  \fill[white] (a) circle(15pt);
  \fill[white] (b) circle(15pt);
  \fill[white] (c) circle(15pt);
  \fill[white] (d) circle(15pt);
  \node(a) at (0,0) {$6$};
  \draw (a) circle(15pt);
  \node(b) at (0,-2) {$3$};
  \draw (b)[line width = 3] circle(15pt);
  \node(c) at (2,-1) {$-4$};
  \draw (c) circle(15pt);
  \node(d) at (4,-1) {$7$};
  \draw (d) circle(15pt);
  \draw[->] (5,-1)--node[above]{$b$}(6,-1);
   \end{scope}
\begin{scope}[yshift = -5cm, xshift = 7cm]
  \node(a) at (0,0) {};
  \node(b) at (0,-2) {};
  \node(c) at (2,-1){};
  \node(d) at (4,-1) {};
  \draw (0,0)--(0,-2);
  \draw (0,0) --(2,-1);
  \draw (0,-2)--(2,-1);
  \draw (2,-1)--node[above]{$\infty$}(4,-1);
  \fill[white] (a) circle(15pt);
  \fill[white] (b) circle(15pt);
  \fill[white] (c) circle(15pt);
  \fill[white] (d) circle(15pt);
  \node(a) at (0,0) {$9$};
  \draw[line width = 3] (a) circle(15pt);
  \node(b) at (0,-2) {$-3$};
  \draw (b) circle(15pt);
  \node(c) at (2,-1) {$-1$};
  \draw (c) circle(15pt);
  \node(d) at (4,-1) {$7$};
  \draw (d) circle(15pt);
  \draw[->] (5,-1)--node[above]{$a$}(6,-1);
   \end{scope}
\begin{scope}[xshift = 14cm,yshift = -5cm]
  \node(a) at (0,0) {};
  \node(b) at (0,-2) {};
  \node(c) at (2,-1){};
  \node(d) at (4,-1) {};
  \draw (0,0)--(0,-2);
  \draw (0,0) --(2,-1);
  \draw (0,-2)--(2,-1);
  \draw (2,-1)--node[above]{$\infty$}(4,-1);
  \fill[white] (a) circle(15pt);
  \fill[white] (b) circle(15pt);
  \fill[white] (c) circle(15pt);
  \fill[white] (d) circle(15pt);
  \node(a) at (0,0) {$-9$};
  \draw (a) circle(15pt);
  \node(b) at (0,-2) {$6$};
  \draw[line width = 3] (b) circle(15pt);
  \node(c) at (2,-1) {$8$};
  \draw (c) circle(15pt);
  \node(d) at (4,-1) {$7$};
  \draw (d) circle(15pt);
  \draw[->] (5,-1)--node[above]{$b$}(6,-1);
   \end{scope}
\begin{scope}[yshift = -10cm, xshift = 7cm]
  \node(a) at (0,0) {};
  \node(b) at (0,-2) {};
  \node(c) at (2,-1){};
  \node(d) at (4,-1) {};
  \draw (0,0)--(0,-2);
  \draw (0,0) --(2,-1);
  \draw (0,-2)--(2,-1);
  \draw (2,-1)--node[above]{$\infty$}(4,-1);
  \fill[white] (a) circle(15pt);
  \fill[white] (b) circle(15pt);
  \fill[white] (c) circle(15pt);
  \fill[white] (d) circle(15pt);
  \node(a) at (0,0) {$-3$};
  \draw (a) circle(15pt);
  \node(b) at (0,-2) {$-6$};
  \draw (b) circle(15pt);
  \node(c) at (2,-1) {$14$};
  \draw (c) circle(15pt);
  \node(d) at (4,-1) {$7$};
  \draw (d) circle(15pt);
   \end{scope}
\end{tikzpicture}
\caption{The start position $\m{p}^{babcda}$ is found.}\label{SPE1}
\end{figure}

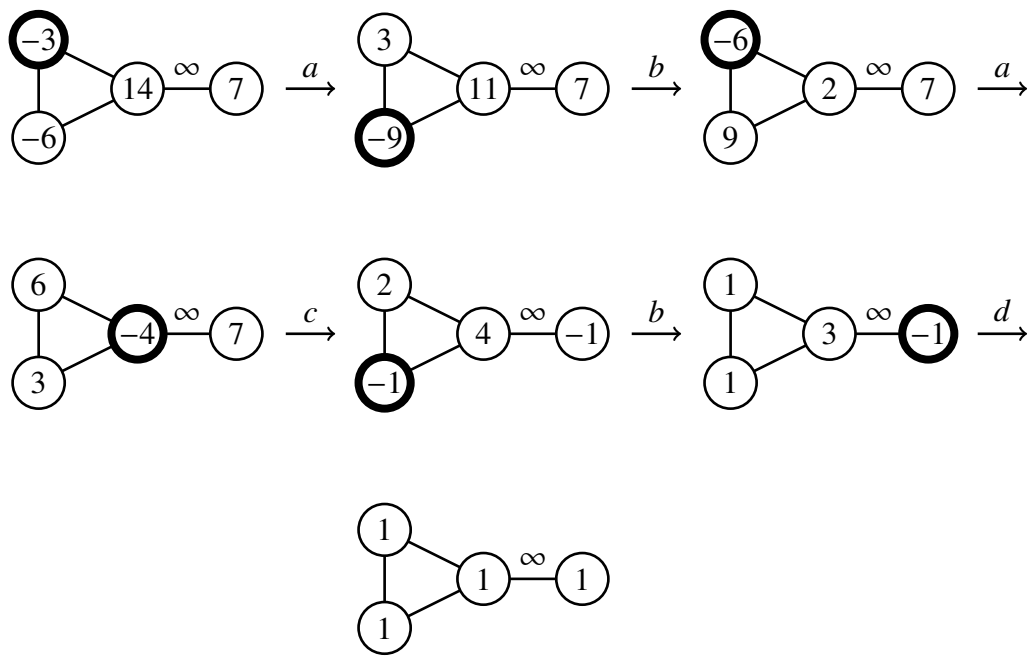
\begin{figure}[h!]
\begin{tikzpicture}[line width =1,scale =0.65,yshift = -1cm]
  \node(a) at (0,0) {};
  \node(b) at (0,-2) {};
  \node(c) at (2,-1){};
  \node(d) at (4,-1) {};
  \draw (0,0)--(0,-2);
  \draw (0,0) --(2,-1);
  \draw (0,-2)--(2,-1);
  \draw (2,-1)--node[above]{$\infty$}(4,-1);
  \fill[white] (a) circle(15pt);
  \fill[white] (b) circle(15pt);
  \fill[white] (c) circle(15pt);
  \fill[white] (d) circle(15pt);
  \node(a) at (0,0) {$-3$};
  \draw[line width = 3] (a) circle(15pt);
  \node(b) at (0,-2) {$-6$};
  \draw (b) circle(15pt);
  \node(c) at (2,-1) {$14$};
  \draw (c) circle(15pt);
  \node(d) at (4,-1) {$7$};
  \draw (d) circle(15pt);
  \draw[->] (5,-1)--node[above]{$a$}(6,-1);
  \begin{scope}[xshift = 7cm]
  \node(a) at (0,0) {};
  \node(b) at (0,-2) {};
  \node(c) at (2,-1){};
  \node(d) at (4,-1) {};
  \draw (0,0)--(0,-2);
  \draw (0,0) --(2,-1);
  \draw (0,-2)--(2,-1);
  \draw (2,-1)--node[above]{$\infty$}(4,-1);
  \fill[white] (a) circle(15pt);
  \fill[white] (b) circle(15pt);
  \fill[white] (c) circle(15pt);
  \fill[white] (d) circle(15pt);
  \node(a) at (0,0) {$3$};
  \draw (a) circle(15pt);
  \node(b) at (0,-2) {$-9$};
  \draw[line width = 3] (b) circle(15pt);
  \node(c) at (2,-1) {$11$};
  \draw (c) circle(15pt);
  \node(d) at (4,-1) {$7$};
  \draw (d) circle(15pt);
  \draw[->] (5,-1)--node[above]{$b$}(6,-1);
   \end{scope}
\begin{scope}[xshift = 14cm]
  \node(a) at (0,0) {};
  \node(b) at (0,-2) {};
  \node(c) at (2,-1){};
  \node(d) at (4,-1) {};
  \draw (0,0)--(0,-2);
  \draw (0,0) --(2,-1);
  \draw (0,-2)--(2,-1);
  \draw (2,-1)--node[above]{$\infty$}(4,-1);
  \fill[white] (a) circle(15pt);
  \fill[white] (b) circle(15pt);
  \fill[white] (c) circle(15pt);
  \fill[white] (d) circle(15pt);
  \node(a) at (0,0) {$-6$};
  \draw[line width = 3] (a) circle(15pt);
  \node(b) at (0,-2) {$9$};
  \draw (b) circle(15pt);
  \node(c) at (2,-1) {$2$};
  \draw (c) circle(15pt);
  \node(d) at (4,-1) {$7$};
  \draw (d) circle(15pt);
  \draw[->] (5,-1)--node[above]{$a$}(6,-1);
   \end{scope}
\begin{scope}[yshift = -5cm]
  \node(a) at (0,0) {};
  \node(b) at (0,-2) {};
  \node(c) at (2,-1){};
  \node(d) at (4,-1) {};
  \draw (0,0)--(0,-2);
  \draw (0,0) --(2,-1);
  \draw (0,-2)--(2,-1);
  \draw (2,-1)--node[above]{$\infty$}(4,-1);
  \fill[white] (a) circle(15pt);
  \fill[white] (b) circle(15pt);
  \fill[white] (c) circle(15pt);
  \fill[white] (d) circle(15pt);
  \node(a) at (0,0) {$6$};
  \draw (a) circle(15pt);
  \node(b) at (0,-2) {$3$};
  \draw (b) circle(15pt);
  \node(c) at (2,-1) {$-4$};
  \draw[line width = 3] (c) circle(15pt);
  \node(d) at (4,-1) {$7$};
  \draw (d) circle(15pt);
  \draw[->] (5,-1)--node[above]{$c$}(6,-1);
   \end{scope}
\begin{scope}[yshift = -5cm, xshift = 7cm]
  \node(a) at (0,0) {};
  \node(b) at (0,-2) {};
  \node(c) at (2,-1){};
  \node(d) at (4,-1) {};
  \draw (0,0)--(0,-2);
  \draw (0,0) --(2,-1);
  \draw (0,-2)--(2,-1);
  \draw (2,-1)--node[above]{$\infty$}(4,-1);
  \fill[white] (a) circle(15pt);
  \fill[white] (b) circle(15pt);
  \fill[white] (c) circle(15pt);
  \fill[white] (d) circle(15pt);
  \node(a) at (0,0) {$2$};
  \draw(a) circle(15pt);
  \node(b) at (0,-2) {$-1$};
  \draw[line width = 3] (b) circle(15pt);
  \node(c) at (2,-1) {$4$};
  \draw (c) circle(15pt);
  \node(d) at (4,-1) {$-1$};
  \draw (d) circle(15pt);
  \draw[->] (5,-1)--node[above]{$b$}(6,-1);
   \end{scope}
\begin{scope}[xshift = 14cm,yshift = -5cm]
  \node(a) at (0,0) {};
  \node(b) at (0,-2) {};
  \node(c) at (2,-1){};
  \node(d) at (4,-1) {};
  \draw (0,0)--(0,-2);
  \draw (0,0) --(2,-1);
  \draw (0,-2)--(2,-1);
  \draw (2,-1)--node[above]{$\infty$}(4,-1);
  \fill[white] (a) circle(15pt);
  \fill[white] (b) circle(15pt);
  \fill[white] (c) circle(15pt);
  \fill[white] (d) circle(15pt);
  \node(a) at (0,0) {$1$};
  \draw (a) circle(15pt);
  \node(b) at (0,-2) {$1$};
  \draw (b) circle(15pt);
  \node(c) at (2,-1) {$3$};
  \draw (c) circle(15pt);
  \node(d) at (4,-1) {$-1$};
  \draw[line width = 3] (d) circle(15pt);
  \draw[->] (5,-1)--node[above]{$d$}(6,-1);
   \end{scope}
\begin{scope}[yshift = -10cm,xshift = 7cm]
  \node(a) at (0,0) {};
  \node(b) at (0,-2) {};
  \node(c) at (2,-1){};
  \node(d) at (4,-1) {};
  \draw (0,0)--(0,-2);
  \draw (0,0) --(2,-1);
  \draw (0,-2)--(2,-1);
  \draw (2,-1)--node[above]{$\infty$}(4,-1);
  \fill[white] (a) circle(15pt);
  \fill[white] (b) circle(15pt);
  \fill[white] (c) circle(15pt);
  \fill[white] (d) circle(15pt);
  \node(a) at (0,0) {$1$};
  \draw (a) circle(15pt);
  \node(b) at (0,-2) {$1$};
  \draw (b) circle(15pt);
  \node(c) at (2,-1) {$1$};
  \draw (c) circle(15pt);
  \node(d) at (4,-1) {$1$};
  \draw (d) circle(15pt);
   \end{scope}
\end{tikzpicture}
\caption{We are firing at each step the minimal negative node.}\label{NFE1}
\end{figure}

\end{document}